\documentclass[12pt]{amsart}%
\usepackage{amsmath}
\usepackage{amsfonts}
\usepackage{amssymb}
\usepackage{graphicx}
\usepackage{comment}
\usepackage[margin=1.5 in]{geometry}
\usepackage{setspace}

\usepackage[utf8]{inputenc}
\usepackage{amsmath,amsthm, amssymb}
\usepackage{pdfrender,xcolor}
\usepackage{tcolorbox}%
\providecommand{\U}[1]{\protect\rule{.1in}{.1in}}
\providecommand{\U}[1]{\protect\rule{.1in}{.1in}}
\providecommand{\U}[1]{\protect\rule{.1in}{.1in}}
\providecommand{\U}[1]{\protect\rule{.1in}{.1in}}
\newtheorem{theorem}{Theorem}

\newtheorem{condition}[theorem]{Condition}

\newtheorem{example}[theorem]{Example}

\newtheorem{lemma}[theorem]{Lemma}

\newtheorem{proposition}[theorem]{Proposition}
\newtheorem{remark}[theorem]{Remark}

\allowdisplaybreaks[1]

\theoremstyle{definition}
\numberwithin{equation}{section}

\newcommand{\resumename}{R\'esum\'e}

\begin{document}
\date{\today}
\title{Compactly supported bounded frames on Lie groups}
\author{Vignon Oussa}
\maketitle

\begin{abstract}
Let $G=NH$ be a Lie group where $N,H$ are closed connected subgroups of $G,$
and $N$ is an exponential solvable Lie group which is normal in $G.$ Suppose
furthermore that $N$ admits a unitary character $\chi_{\lambda}$ corresponding
to a linear functional $\lambda$ of its Lie algebra. We assume that the map
$h\mapsto Ad\left(  h^{-1}\right)  ^{\ast}\lambda$ defines an immersion at the
identity of $H$. Fixing a Haar measure on $H,$ we consider the unitary
representation $\pi$ of $G$ obtained by inducing $\chi_{\lambda}.$ This
representation which is realized as acting in $L^{2}\left(  H,d\mu_{H}\right)
$ is generally not irreducible, and we do not assume that it satisfies any
integrability condition. One of our main results establishes the existence of
a countable set $\Gamma\subset G$ and a function $\mathbf{f}\in L^{2}\left(
H,d\mu_{H}\right)  $ which is compactly supported and bounded such that
$\left\{  \pi\left(  \gamma\right)  \mathbf{f}:\gamma\in\Gamma\right\}  $ is a
frame. Additionally, we prove that $\mathbf{f}$ can be constructed to be
continuous. In fact, $\mathbf{f}$ can be taken to be as smooth as desired. Our
findings extend the work started in \cite{oussa2018frames} to the more general
case where $H$ is any connected Lie group. We also solve a problem left open
in \cite{oussa2018frames}. Precisely, we prove that in the case where $H$ is
an exponential solvable group, there exist a continuous (or smooth) function
$\mathbf{f}$ and a countable set $\Gamma$ such that $\left\{  \pi\left(
\gamma\right)  \mathbf{f}:\gamma\in\Gamma\right\}  $ is a Parseval frame.
Since the concept of well-localized frames is central to time-frequency
analysis, wavelet, shearlet and generalized shearlet theories, our results are
relevant to these topics and our approach leads to new constructions which
bear potential for applications. Moreover, our work sets itself apart from
other discretization schemes in many ways. (1) We give an explicit
construction of Hilbert frames and Parseval frames generated by bounded and
compactly supported windows. (2) We provide a systematic method that can be
exploited to compute frame bounds for our constructions. (3) We make no
assumption on the irreducibility or integrability of the representations of interest.

\end{abstract}
\tableofcontents

\onehalfspacing




\section{Introduction and overview of the work}

Let $H$ be a Lie group endowed with a fixed left Haar measure $d\mu_{H}.$ A
frame \cite{MR1946982} for the Hilbert space $\mathfrak{H}=L^{2}\left(
H,d\mu_{H}\right)  $ is a collection of vectors $\left\{  \mathbf{f}_{k}:k\in
I\right\}  $ such that there exist positive constants $A\leq B$ where
\begin{equation}
A\left\Vert \mathbf{g}\right\Vert _{\mathfrak{H}}^{2}\leq\sum_{k\in
I}\left\vert \left\langle \mathbf{g},\mathbf{f}_{k}\right\rangle
_{\mathfrak{H}}\right\vert ^{2}\leq B\left\Vert \mathbf{g}\right\Vert
_{\mathfrak{H}}^{2} \label{ineq}%
\end{equation}
for all $\mathbf{g}\in\mathfrak{H}.$ Note that the inequality (\ref{ineq})\ is
a relaxation of the Parseval formula known for orthonormal bases. The
constants $A,B$ are called the frame bounds. The largest possible lower frame
bound is called the optimal lower bound and the smallest possible upper-bound
is the optimal upper frame bound. In the case where $A=B=1$, we say that
$\left\{  \mathbf{f}_{k}:k\in I\right\}  $ is a Parseval frame. Every frame
$\left\{  \mathbf{f}_{k}:k\in I\right\}  $ for a Hilbert space $\mathfrak{H}$
gives rise to a basis-like expansion of the type $\mathbf{g=}\sum_{k\in
I}a_{k}\left(  \mathbf{g}\right)  \mathbf{f}_{k}$ for every vector
$\mathbf{g}$ in $\mathfrak{H.}$ Moreover, the linear functionals $a_{k}%
$\textbf{ }can be selected to be bounded on $\mathfrak{H}$. As such, for every
$\mathbf{g}$ in $\mathfrak{H,}$ there exist vectors $\left\{  \mathbf{h}%
_{k}:k\in I\right\}  $ in $\mathfrak{H}$ such that $\mathbf{g=}\sum_{k\in
I}\left\langle \mathbf{g,h}_{k}\right\rangle \mathbf{f}_{k}.$ There is indeed
a canonical choice for the system $\left\{  \mathbf{h}_{k}:k\in I\right\}  $
for which the expansion above is unconditional. Precisely, for a frame
$\left\{  \mathbf{f}_{k}:k\in I\right\}  ,$ the associated frame operator
\[
S:\mathbf{g}\mapsto S\left(  \mathbf{g}\right)  =\sum_{k\in I}\left\langle
\mathbf{g},\mathbf{f}_{k}\right\rangle _{\mathfrak{H}}\mathbf{f}_{k}%
\]
is bounded, self-adjoint, positive and invertible on $\mathfrak{H}$ and every
vector $\mathbf{g}$ admits an expansion of the form $\mathbf{g=}\sum_{k\in
I}\left\langle \mathbf{g},S^{-1}\mathbf{f}_{k}\right\rangle _{\mathfrak{H}%
}\mathbf{f}_{k}.$ The frame operator and its inverse are topological
isomorphisms of $\mathfrak{H}$ and the series $\sum_{k\in I}\left\langle
\mathbf{g},S^{-1}\mathbf{f}_{k}\right\rangle _{\mathfrak{H}}\mathbf{f}_{k}$
converges unconditionally in the norm of $\mathfrak{H}$ for each $\mathbf{g}$
in $\mathfrak{H}.$ In fact, the sequence $\left\{  S^{-1}\mathbf{f}_{k}:k\in
I\right\}  $ is called the canonical dual frame of $\left\{  \mathbf{f}%
_{k}:k\in I\right\}  .$ Moreover, each sequence $\left(  \left\langle
\mathbf{g},S^{-1}\mathbf{f}_{k}\right\rangle _{\mathfrak{H}}\right)  _{k\in
I}$ is square-summable. However, the coefficients appearing in the expansion
$\sum_{k\in I}\left\langle \mathbf{g},S^{-1}\mathbf{f}_{k}\right\rangle
_{\mathfrak{H}}\mathbf{f}_{k}$ need not be unique. In the case where $\left\{
\mathbf{f}_{k}:k\in I\right\}  $ is a Parseval frame, we have a much simpler
reconstruction procedure. The frame operator coincides with the identity map
and as a result,
\[
\mathbf{g=}\sum_{k\in I}\left\langle \mathbf{g},\mathbf{f}_{k}\right\rangle
_{\mathfrak{H}}\mathbf{f}_{k}\text{ for all }\mathbf{g}\in\mathfrak{H.}%
\]

It is well-known that the Schr\"{o}dinger representations of the 3-dimensional
Heisenberg Lie group provide the group-theoretic foundation for time-frequency
analysis
\cite{grochenig2013foundations,christensen2016introduction,heil2010basis}. The
3-dimensional Heisenberg group is often realized as a semi-direct product $%
\mathbb{R}
^{2}\rtimes%
\mathbb{R}
$ equipped with the following non-commutative operation%
\[
\left(  v,t\right)  \left(  w,s\right)  =\left(  v+\left[
\begin{array}
[c]{cc}%
1 & t\\
0 & 1
\end{array}
\right]  w,t+s\right)  .
\]
Three elementary classes of operators occur in the Schr\"{o}dinger
representations: scalar multiplications, modulation operators (phase
translations) and translations. Although the role played by the operators
acting by scalar multiplication is negligible, the combined actions of
modulations and translations can be discretized to construct frames and
orthonormal bases for $L^{2}(\mathbb{R})$. Gabor analysis offers a rigorous
study of the various ways by which one can select a window function and a
countable subset of the Heisenberg group to construct frames. Similarly, the
ax+b group plays a significant role in classical wavelet analysis. Like the
Heisenberg group, the ax+b group is also a solvable Lie group. However, unlike
the Heisenberg group, it is not a nilpotent group. Very often, this affine
group is realized as a semi-direct product $%
\mathbb{R}
\rtimes%
\mathbb{R}
$ endowed with the multiplication
\[
\left(  x,a\right)  \left(  y,b\right)  =\left(  x+e^{a}y,a+b\right)  .
\]
Up to unitary equivalence, the ax+b group has two infinite-dimensional
irreducible representations. These representations may be realized as acting
in $L^{2}(\mathbb{R})$ as follows
\[
\pi_{\pm}\left(  x,a\right)  \mathbf{f}\left(  t\right)  =e^{\pm2\pi ie^{-t}%
x}\mathbf{f}\left(  t-a\right)  .
\]
Wavelet analysis provides means by which we can discretize the representations
$\pi_{\pm}$ (or their direct sum) to construct nice frames (we will say more
about this later.) A common feature of the examples of the Heisenberg group
and the ax+b group worth noting is the following: the irreducible
representations giving rise to wavelets and Gabor frames are all monomial
(induced by characters.) Precisely, they are obtained by inducing a unitary
character of a normal closed subgroup. It turns out that for a surprisingly
large class of solvable Lie groups, similar constructions are possible
\cite{oussa2018frames,grochenig2017orthonormal}. One may then ask the
following: to which extent can we generalize the results in
\cite{oussa2018frames,grochenig2017orthonormal} in a unified manner? The main
purpose of this work is to address this question.

Let $H$ be a connected Lie group endowed with a fixed left Haar measure
$d\mu_{H}.$ Let us suppose that $\pi$ is a strongly continuous unitary
representation of a Lie group $G$ acting in the Hilbert space $\mathfrak{H}%
=L^{2}\left(  H,d\mu_{H}\right)  .$ We investigate conditions under which, it
is possible to find a discrete subset $\Gamma$ of $G$ and a function
$\mathbf{f}$ such that the following holds. $\mathbf{f}$ is bounded and
compactly supported on $H$ and the system $\left\{  \pi\left(  \gamma\right)
\mathbf{f}:\gamma\in\Gamma\right\}  $ is either a frame or a Parseval frame or
an orthonormal basis. We are especially interested in finding conditions under
which the window function $\mathbf{f}$ is either continuous or smooth. In a
nutshell, we seek to construct `nice' frames, and when such a frame exists, we
shall call it an $H$-localized $\pi$-frame.

Suppose $L$ is the left regular representation of $H$ acting in $\mathfrak{H}%
,$ and $H$ is not discrete. If $\mathbf{f}$ is compactly supported and bounded
then for any countable set $\Gamma\subset H,$ it is known that the system
$\left\{  \pi\left(  \gamma\right)  \mathbf{f}:\gamma\in\Gamma\right\}  $ can
never be a frame for $\mathfrak{H}$ \cite{fuhr2018groups}. This observation
excludes the left regular representation as a viable choice for the
construction of nice frames on $H.$

In the situation where $\pi$ is an irreducible representation which is also
integrable, the coorbit theory
\cite{MR1021139,grochenig1991describing,fuhr2015coorbit} developed by
Feichtinger and Gr\"{o}chenig has proved to be a powerful discretization
scheme for the construction of Hilbert space frames and atoms for other Banach
spaces satisfying certain regular conditions.

Olafsson and Christensen recently introduced a generalization of coorbit
theory in which, the integrability and irreducibility conditions of coorbit
theory were removed \cite{christensen2009examples,christensen2011coorbit}.
Their theory hinges on the existence of a cyclic vector satisfying a
reproducing formula equation together with other technical assumptions. Other
generalizations of coorbit theory can be found in
\cite{dahlke2017coorbit,fornasier2005continuous,rauhut2011generalized}.

Our work sets itself apart from other discretization schemes in several ways.

\begin{itemize}
\item We make no assumption on the irreducibility or integrability of the
representations of interest.

\item We are mainly interested in the explicit construction of Hilbert space
frames, Parseval frames and orthonormal bases generated by bounded and
compactly supported windows.

\item We are also interested in providing a systematic method that can be
exploited to compute (optimal) frame bounds for our constructions.
\end{itemize}

Let $G$ be a Lie group such that $G=NH,$ $H,N$ are connected closed subgroups,
and $N$ is simply connected and normal in $G$. Assume furthermore that $\pi$
is a unitary representation of $G$ \cite{folland2016course} obtained by
inducing a unitary character of $N$ and is realized as acting in the Hilbert
space $L^{2}\left(  H,d\mu_{H}\right)  .$ As a starting point, we consider a
unitary character $\chi$ of $N$ associated to a linear functional $\lambda$ of
the Lie algebra of $N.$ The conjugation action of $H$ on $N$ gives rise to a
smooth action $\star$ of $H$ on the linear functionals of the Lie algebra of
$N$ . This action plays a decisive role in our work. We prove that if the map
$h\mapsto h\star\lambda$ defines an immersion at the identity of the group $H$
(see Theorem \ref{main}) then $H$-localized $\pi$-frames can be constructed
quite explicitly. These types of frames naturally arise in the context of
shearlet, time-frequency analyses, and wavelet theory
\cite{kutyniok2012introduction,kutyniok2012shearlets,
kutyniok2011compactly,Gr,fuhr2016vanishing,fuhr2015coorbit,fuhr2017simplified}
where frames generated by functions satisfying certain regular conditions
(such as fast decay and various degrees of smoothness) are often sought after.
As such, the present project, which was initiated in \cite{oussa2018frames}
provides a unified discretization scheme relevant to these topics.
Furthermore, our main results (see Theorem \ref{main} and Theorem \ref{PFEXP})
extend the work of \cite{oussa2018frames} to a significantly broader class of
representations (see Proposition \ref{embed} for an example). While the
central results of \cite{oussa2018frames} only focus on cases where the
conormal subgroup $H$ of $G$ is completely solvable, we allow $H$ to be any
connected Lie group here (see for example, Proposition \ref{embed}.)

\subsection{Assumptions\label{assump}}

Before we present our main results, let us first formally introduce some
assumptions. Let $H$ be a connected Lie group with Lie algebra $\mathfrak{h.}$
Let us also assume that there exists a Lie algebra $\mathfrak{g=n\oplus h}$
such that $\mathfrak{n}$ is an exponential solvable ideal of $\mathfrak{g}$
satisfying $\dim\left(  \mathfrak{n}\right)  =n \geq\dim\left(  \mathfrak{h}%
\right)  =r .$ Put%
\[
\mathfrak{n}=\sum_{k=1}^{n}\mathbb{R}X_{k}\text{, }\mathfrak{h}=\sum_{k=1}%
^{r}\mathbb{R}A_{k}\text{ and }N=\exp\left(  \mathfrak{n}\right)  .
\]
Furthermore, we assume that $N$ and $H$ are closed subgroups of $G.$ Note that
by assumption, the exponential map defines a global diffeomorphism between
$\mathfrak{n}$ and $N.$ As such, $N$ is a simply connected normal
(exponential) solvable subgroup of $G.$

Let
\[
\lambda\in\mathfrak{n}^{\ast}=\sum_{k=1}^{n}\mathbb{R}X_{k}^{\ast}%
\]
be a linear functional of $\mathfrak{n}$ where $\left\{  X_{1}^{\ast}%
,X_{2}^{\ast},\cdots,X_{n}^{\ast}\right\}  $ is a basis for the dual vector
space $\mathfrak{n}^{\ast}$ satisfying $\left\langle X_{j}^{\ast}%
,X_{k}\right\rangle =\delta_{j,k}$ for $1\leq j,k\leq n.$ We also assume that
the derived ideal
\[
\left[  \mathfrak{n,n}\right]  =%
\mathbb{R}
\text{-span}\left\{  \left[  X,Y\right]  :X,Y\in\mathfrak{n}\right\}
\]
is contained in the kernel of $\lambda.$ As such, the linear functional
$\lambda$ determines a unitary character of the normal subgroup of $N$ defined
by $\chi_{\lambda}\left(  \exp X\right)  =e^{2\pi i\left\langle \lambda
,X\right\rangle }$ for $\exp\left(  X\right)  \in N.$ Next, we consider the
induced representation
\begin{equation}
\pi=\mathrm{ind}_{N}^{NH}\left(  \chi_{\lambda}\right)  \label{pi}%
\end{equation}
realized as acting on the Hilbert space $\mathfrak{H}=L^{2}\left(  H,d\mu
_{H}\right)  $ as follows. Given $x\in G,$ we define the endomorphism
$Ad\left(  s\right)  $ as the differential of the smooth map $y\mapsto
sys^{-1}$ at the neutral element of $G.$ Next, for $\mathbf{f}\in\mathfrak{H}%
$, $x\in N,$ and $z\in H,$ we have%
\begin{equation}
\left[  \pi\left(  x\right)  \right]  \mathbf{f}\left(  h\right)  =e^{2\pi
i\left\langle Ad\left(  h^{-1}\right)  ^{\ast}\lambda,\log\left(  x\right)
\right\rangle }\mathbf{f}\left(  h\right)  \text{ and }\left[  \pi\left(
z\right)  \right]  \mathbf{f}\left(  h\right)  =\mathbf{f}\left(
z^{-1}h\right)  . \label{irred}%
\end{equation}
The reader who is interested in learning about the fundamental principles of
induced representations is invited to peruse Chapter 6 in the textbook of
Folland \cite{folland2016course}.

On one hand, we remark that given $x\in N,$ the operator $\pi\left(  x\right)
$ acts in a way which resembles a modulation action. We call the action
induced by such an operator a generalized modulation action. On the other
hand, the operators $\pi\left(  z\right)  $ (where $z\in H$) act on
$\mathfrak{H}$ by left translations and the combined action of the normal and
conormal parts of the group gives an action resembling that of time-frequency
translations \cite{MR1946982}. Perhaps, it is also worth highlighting that all
irreducible representations of groups connected to Gabor theory, wavelet and
shearlet analyses naturally arise in this fashion (see \cite[Section
1.4]{oussa2018frames}).

We refer the reader who might be intrigued by the size of the class of groups
and representations satisfying the properties described above to Proposition
\ref{embed}. It is proved in Proposition \ref{embed} that for any given
connected matrix group $H$, there exists a matrix group $G=NH$ and a unitary
representation $\pi$ satisfying all conditions listed above.

\subsection{Main theorems}

The main results of our investigation are summarized in the following theorems.

\begin{theorem}
\label{main} Let $N,H$ and $\pi$ be as described in (\ref{irred}). Assume that
the map $h\mapsto Ad(h^{-1})^{\ast}\lambda$ is an immersion at the identity of
$H$. Then there exists a relatively compact neighborhood $\mathcal{O}$ of the
identity in $H$ such that if $\mathbf{f}$ is a continuous function supported
on $\mathcal{O}$, then there exists a countable set $\Gamma\subset G$ such
that $\left\{  \pi\left(  \gamma\right)  \mathbf{f}:\gamma\in\Gamma\right\}  $
is a frame for $L^{2}\left(  H,d\mu_{H}\right)  .$ Moreover, $\mathbf{f}$ can
be chosen to be as smooth as desired.
\end{theorem}

We remark that in Theorem \ref{main}, the set $\Gamma$ depends on a choice of
$\mathbf{f}$ and $\mathcal{O}.$ Moreover, the following establishes the
existence of a Parseval frame generated by a continuous and compactly
supported function when $H$ is exponential and solvable. This does close a
significant gap in our previous work \cite{oussa2018frames}.

\begin{theorem}
\label{PFEXP}Let $N,H$ and $\pi$ be as described in (\ref{irred}). If the map
$h\mapsto Ad(h^{-1})^{\ast}\lambda$ is an immersion at the identity of $H$ and
if $H$ is an exponential solvable Lie group then there exist a relatively
compact neighborhood $\mathcal{O}$ of the identity in $H,$ a continuous
function $\mathbf{f}$ supported on $\mathcal{O}$ and a countable subset
$\Gamma$ of $G$ such that $\left\{  \pi\left(  \gamma\right)  \mathbf{f}%
:\gamma\in\Gamma\right\}  $ is a Parseval frame.
\end{theorem}

A few additional observations are in order.

\begin{remark}
\text{ }

\begin{enumerate}
\item In a systematic fashion, every frame $\left\{  \pi\left(  \gamma\right)
\mathbf{f}:\gamma\in\Gamma\right\}  $ for the Hilbert space $L^{2}\left(
H,d\mu_{H}\right)  $ gives rise to new frames generated by a unitary
representation of $G$ as follows. Let $\alpha$ be a Lie group automorphism of
$G.$ According to \cite[Lemma 2.1.3]{corwin2004representations}, there exists
a unitary map
\begin{equation}
Q:L^{2}\left(  H,d\mu_{H}\right)  \rightarrow L^{2}\left(  NH/\alpha
^{-1}\left(  N\right)  ,d\mu_{NH/\alpha^{-1}\left(  N\right)  }\right)
\label{Q}%
\end{equation}
such that
\[
Q\left[  \mathrm{Ind}_{N}^{NH}\left(  \chi_{\lambda}\right)  \left(
\alpha\left(  x\right)  \right)  \right]  \mathbf{f}=\left[  \mathrm{Ind}%
_{\alpha^{-1}\left(  N\right)  }^{NH}\left(  \chi_{\lambda}\circ\alpha\right)
\left(  x\right)  \right]  Q\mathbf{f}%
\]
for all $x\in G$ and for all $\mathbf{f}\in L^{2}\left(  H,d\mu_{H}\right)  .$
An explicit construction of a unitary map $Q$ intertwining the representations
above is described in the proof of \cite[ Lemma 2.1.3]%
{corwin2004representations}. In fact, if $\pi$ is an irreducible
representation of $G,$ according to Schur's lemma, the unitary operator
described in the proof of \cite[ Lemma 2.1.3]{corwin2004representations} is
unique up to multiplication by a constant. In the definition of $Q,$ (see
(\ref{Q})) the measure $d\mu_{NH/\alpha^{-1}\left(  N\right)  }$ is a quasi
$G$-invariant measure on a cross-section for $NH/\alpha^{-1}\left(  N\right)
$ \cite[Section 6.1]{folland2016course}. Letting
\[
\pi_{\lambda^{\left(  \alpha\right)  }}=\mathrm{Ind}_{\alpha^{-1}\left(
N\right)  }^{NH}\left(  \chi_{\lambda}\circ\alpha\right)  \text{ and }%
\pi_{\lambda_{\left(  \alpha\right)  }}=\mathrm{Ind}_{N}^{NH}\left(
\chi_{\lambda}\circ\alpha\right)  ,
\]
the following is immediate. (1) $\left\{  \pi\left(  \gamma\right)
\mathbf{f}:\gamma\in\Gamma\right\}  $ is a frame if and only if
\[
\left\{  \pi_{\lambda^{\left(  \alpha\right)  }}\left(  \gamma\right)
Q\mathbf{f}:\gamma\in\alpha^{-1}\left(  \Gamma\right)  \right\}
\]
is a frame for $L^{2}\left(  NH/\alpha^{-1}\left(  N\right)  ,d\mu
_{NH/\alpha^{-1}\left(  N\right)  }\right)  .$ (2) If $\alpha$ is an
automorphism of $G$ fixing the normal subgroup $N$ then $\left\{  \pi\left(
\gamma\right)  \mathbf{f}:\gamma\in\Gamma\right\}  $ is a frame for
$L^{2}\left(  H,d\mu_{H}\right)  $ if and only if the collection $\left\{
\pi_{\lambda_{\left(  \alpha\right)  }}\left(  \gamma\right)  Q\mathbf{f}%
:\gamma\in\alpha^{-1}\left(  \Gamma\right)  \right\}  $ is a frame for
$L^{2}\left(  H,d\mu_{H}\right)  .$

\item Suppose that $\left\{  \pi\left(  \gamma\right)  \mathbf{f}:\gamma
\in\Gamma\right\}  $ is a frame for the Hilbert space $L^{2}\left(  H,d\mu
_{H}\right)  .$ If $Q$ is a bounded and invertible linear operator commuting
with all operators in the set $\left\{  \pi\left(  \gamma\right)  :\gamma
\in\Gamma\right\}  $ then clearly, $\left\{  \pi\left(  \gamma\right)
Q\mathbf{f}:\gamma\in\Gamma\right\}  $ is also a frame for $L^{2}\left(
H,d\mu_{H}\right)  .$

\item The mapping $h\mapsto Ad\left(  h^{-1}\right)  ^{\ast}$ is a linear
representation of $H$ acting on the dual of the Lie algebra of $N.$ Moreover,
the stabilizer subgroup
\[
H_{\lambda}=\left\{  h\in H:Ad\left(  h^{-1}\right)  ^{\ast}\lambda
=\lambda\right\}
\]
is a closed subgroup of $H$ and $h\mapsto Ad(h^{-1})^{\ast}\lambda$ is an
immersion at the identity of $H$ if and only if $H_{\lambda}$ is discrete (see
Proposition \ref{discrete}.)
\end{enumerate}
\end{remark}

Intuitively, our strategy in proving our main results is two-fold. First, we
look for a set $M$ inside $N$ containing a discrete set $\Gamma_{N}$ such that
the corresponding operators applied to a suitable function supported on a
neighborhood of the identity in $H$ generate what we shall call a `local
frame'. The construction of local frames follows from the fact that $H$ is
locally Euclidean and that under suitable change of variables, the action of
$N$ gives rise to some local Fourier series. It is important to note that $M$
is generally a proper sub-manifold of $N.$ In fact, the discretization of $M$
for the construction of a local frame is not apparent and does require careful
work. Secondly, via $\pi,$ $H$ acts on itself by left translations. For a
suitable countable subset $\Gamma_{H}\subset H,$ the operators $\pi
(\gamma),\gamma\in\Gamma_{H}$ give a covering of $H$ allowing us to construct
nice frames for $L^{2}(H)$ from our so-called local frames.

In Section \ref{proofs}, we shall present a number of intermediate facts
leading to the proof of our main results (Theorem \ref{main} and Theorem
\ref{PFEXP}). The third section of the work contains explicit examples showing
how our scheme can be exploited for the construction of frames, Parseval
frames and orthonormal bases.

\section{Intermediate results and proofs of Theorem \ref{main} and Theorem
\ref{PFEXP}}

\label{proofs}

\subsection{Preliminaries}

Let us start by fixing our notation.

\begin{itemize}
\item Let $T$ be a linear operator acting on a vector. $T^{\ast}$ stands for
the adjoint of the linear operator $T$ and the transpose of a matrix $A$ is
denoted $A^{T}.$

\item Given a vector space $V,$ the max-norm of an arbitrary vector $v\in V$
is $\left\Vert v\right\Vert _{\max}=\max_{k}\left\vert v_{k}\right\vert .$

\item Given a Lebesgue measurable subset $E$ of $\mathbb{R}^{d}$, $\left\vert
E\right\vert $ stands for the Lebesgue measure of $E$ and the indicator
function of a set $A$ is denoted $1_{A}.$

\item Given a finite set $J$, the number of elements in $J$ is denoted
$\sharp(J).$

\item The trace of a matrix $M$ is denoted $\mathrm{Tr}\left(  M\right)  .$

\item The zero square matrix of order $m$ is often denoted $0_{m}.$

\item The identity endomorphism acting on a vector space $V$ is often denoted
$\mathrm{id}$.
\end{itemize}

Define
\begin{equation}
\beta_{\lambda}:H\rightarrow\beta_{\lambda}\left(  H\right)  \subseteq
\mathfrak{n}^{\ast} \label{Beta}%
\end{equation}
such that
\begin{equation}
\beta_{\lambda}\left(  h\right)  =Ad\left(  h^{-1}\right)  ^{\ast}\lambda.
\label{beta}%
\end{equation}

We make the following additional assumption: The smooth function $h\mapsto
Ad\left(  h^{-1}\right)  ^{\ast}\lambda$ is an immersion at the identity of
$H$. This implies that $n=\dim N\geq r=\dim H.$ We make no assumption on the
irreducibility of $\pi,$ nor shall we assume that $\pi$ is integrable or even
square-integrable (see Example \ref{solv}). In fact, as soon as the structure
constants of the Lie algebra $\mathfrak{g}$ are determined, the immersion
condition imposed on the function $h\mapsto Ad\left(  h^{-1}\right)  ^{\ast
}\lambda$ is easily verified (see Proposition \ref{1}.) To this end, it
suffices to show that the differential of the map $h\mapsto Ad\left(
h^{-1}\right)  ^{\ast}\lambda$ has a trivial null-space at the identity of
$H$. Letting $r=\dim\left(  H\right)  $, there is a smooth chart
\begin{equation}
\left(  \mathcal{O},\varphi\right)  \label{Chart}%
\end{equation}
such that $\mathcal{O}\subset H$ is an open set around the identity of $H$,
and $\varphi:\mathcal{O}\rightarrow\varphi\left(  \mathcal{O}\right)
\subseteq\mathfrak{h}$ is a diffeomorphism taking the identity in $H$ to the
zero element in its Lie algebra such that $\beta_{\lambda}\circ\varphi^{-1}$
has a constant rank equal to $r.$ Since the exponential map defines a local
diffeomorphism between an open set containing the zero vector in
$\mathfrak{h}$ and a suitable open set around the neutral element in $H,$ for
a sufficiently small neighborhood of the identity, the function $\varphi$ can
be taken to be the $\log$ function. Moving forward, the Jacobian of $h\mapsto
Ad\left(  h^{-1}\right)  ^{\ast}\lambda$ in local coordinates takes the form
\begin{equation}
D=\mathrm{Jac}_{\beta_{\lambda}\circ\varphi^{-1}}\left(  0\right)  =\left[
\dfrac{\partial\left(  \beta_{\lambda}\circ\varphi^{-1}\right)  _{j}\left(
A\right)  }{\partial A_{k}}\left(  0\right)  \right]  _{1\leq j\leq n,1\leq
k\leq r}%
\end{equation}
and the following is immediate.

\begin{proposition}
\label{1}$h\mapsto Ad\left(  h^{-1}\right)  ^{\ast}\lambda$ is an immersion at
the identity of $H$ if and only if $\det\left(  D^{T}D\right)  $ is nonzero.
\end{proposition}

\begin{proof}
By definition, $h\mapsto Ad\left(  h^{-1}\right)  ^{\ast}\lambda$ is an
immersion at the identity of $H$ if and only if the matrix $D$ has trivial
kernel. Finally, the observation that $D$ has trivial kernel if and only if
its gramian matrix $D^{T}D$ is invertible gives the desired result.
\end{proof}

\begin{proposition}
\label{2}Let $\omega\in\mathfrak{n}^{\ast}.$ Fixing $\left\{  X_{1}%
,X_{2},\cdots,X_{n}\right\}  $ as a basis for $\mathfrak{n},$ and $\left\{
A_{1},A_{2},\cdots,A_{r}\right\}  $ as a basis for $\mathfrak{h}$, then
\[
D_{\omega}=\left[
\begin{array}
[c]{ccc}%
\left\langle \omega,\left[  X_{1},A_{1}\right]  \right\rangle  & \cdots &
\left\langle \omega,\left[  X_{1},A_{r}\right]  \right\rangle \\
\vdots & \ddots & \vdots\\
\left\langle \omega,\left[  X_{n},A_{1}\right]  \right\rangle  & \cdots &
\left\langle \omega,\left[  X_{n},A_{r}\right]  \right\rangle
\end{array}
\right]
\]
is a matrix representation of $\mathrm{Jac}_{\beta_{\omega}\circ\varphi^{-1}%
}\left(  0\right)  $ with respect to the fixed bases.
\end{proposition}

\begin{proof}
Given $t\in\mathbb{R}$, $A\in\mathfrak{h}$ and $X\in\mathfrak{n,}$ we have
\[
\left\langle \lim_{t\rightarrow0}\frac{\left[  Ad\left(  \exp\left(
-tA\right)  \right)  \right]  ^{\ast}\omega-\omega}{t},X\right\rangle
=\lim_{t\rightarrow0}\frac{\left\langle \left[  Ad\left(  \exp\left(
-tA\right)  \right)  \right]  ^{\ast}\omega-\omega,X\right\rangle }{t}.
\]
Next, since
\[
\left\langle \omega,Ad\left(  \exp\left(  -tA\right)  \right)  X\right\rangle
=\left\langle \omega,\sum_{k=0}^{\infty}ad\left(  -tA\right)  ^{k}%
X\right\rangle
\]
the following holds true
\begin{align*}
\left\langle \lim_{t\rightarrow0}\frac{\left[  Ad\left(  \exp\left(
-tA\right)  \right)  \right]  ^{\ast}\omega-\omega}{t},X\right\rangle  &
=\lim_{t\rightarrow0}\frac{\left\langle \omega,Ad\left(  \exp\left(
-tA\right)  \right)  X\right\rangle }{t}-\lim_{t\rightarrow0}\frac
{\left\langle \omega,X\right\rangle }{t}\\
&  =\lim_{t\rightarrow0}\frac{\left\langle \omega,\sum_{k=1}^{\infty}ad\left(
-tA\right)  ^{k}X\right\rangle }{t}\\
&  =\left(  \ast\right)
\end{align*}
Observing that
\[
\sum_{k=1}^{\infty}ad\left(  -tA\right)  ^{k}X=ad\left(  -tA\right)
X+\sum_{k=2}^{\infty}ad\left(  -tA\right)  ^{k}X,
\]
we obtain
\begin{align*}
\left(  \ast\right)   &  =\left(  \lim_{t\rightarrow0}\frac{\left\langle
\omega,ad\left(  -tA\right)  X\right\rangle }{t}\right)  +\underset
{=0}{\underbrace{\left(  \lim_{t\rightarrow0}\frac{\left\langle \omega
,\sum_{k=2}^{\infty}\left(  -t\right)  ^{k}ad\left(  A\right)  X\right\rangle
}{t}\right)  }}\\
&  =-\left\langle \omega,\left[  A,X\right]  \right\rangle =\left\langle
\omega,\left[  X,A\right]  \right\rangle
\end{align*}
and this gives the desired result.
\end{proof}

\begin{proposition}
The set of linear functionals $\omega$ for which $\beta_{\omega}$ is a local
immersion at the identity in $H$ is a Zariski open subset of $\mathfrak{n}%
^{\ast}.$
\end{proposition}

\begin{proof}
Let $\mathcal{Q}=\left\{  \omega\in\mathfrak{n}^{\ast}:\beta_{\omega}\text{ is
an immersion at the identity in }H\right\}  .$ To prove the stated result, it
suffices to establish that $\mathcal{Q}$ is the complement of a zero set of a
polynomial defined on $\mathfrak{n}^{\ast}.$ Put
\[
D\left(  \omega\right)  =\left[
\begin{array}
[c]{ccc}%
\left\langle \omega,\left[  X_{1},A_{1}\right]  \right\rangle  & \cdots &
\left\langle \omega,\left[  X_{1},A_{r}\right]  \right\rangle \\
\vdots & \ddots & \vdots\\
\left\langle \omega,\left[  X_{n},A_{1}\right]  \right\rangle  & \cdots &
\left\langle \omega,\left[  X_{n},A_{r}\right]  \right\rangle
\end{array}
\right]  \text{ and }p\left(  \omega\right)  =D\left(  \omega\right)  D\left(
\omega\right)  ^{T}.
\]
Clearly $p$ is a polynomial defined on $\mathfrak{n}^{\ast}.$ Appealing to
Proposition \ref{1} and Proposition \ref{2},
\begin{equation}
\mathcal{Q}=\left\{  \omega\in\mathfrak{n}^{\ast}:\det\left(  D\left(
\omega\right)  D\left(  \omega\right)  ^{T}\right)  \neq0\right\}  .
\end{equation}

\end{proof}

Since $n$ is greater than or equal to $r$ and $D$ has an invertible submatrix
of order $r$ then there is at least a subset $J$ of $\left\{  1,2,\cdots
,n\right\}  $ such that $\sharp\left(  J\right)  =r=\dim H$ and the submatrix
of $D$ obtained by retaining all rows of $D$ corresponding to the elements of
$J$ is an invertible square matrix of order $r.$ By the Inverse Function
Theorem (see Theorem 5.11, \cite{lee2003smooth}), there is a linear
projection
\begin{equation}
P=P_{J}:\mathfrak{n\rightarrow n} \label{projection}%
\end{equation}
of rank $r$ (depending on the set $J$) given by
\begin{equation}
P\left(  X_{k}\right)  =\left\{
\begin{array}
[c]{c}%
X_{k}\text{ if }k\in J\\
0\text{ if }k\notin J
\end{array}
\right.
\end{equation}
such that
\[
\Theta_{J,\lambda}=\Theta_{\lambda}=P^{\ast}\beta_{\lambda}\varphi
^{-1}:\varphi\left(  \mathcal{O}\right)  \rightarrow\Theta_{\lambda}\left(
\varphi\left(  \mathcal{O}\right)  \right)
\]
is a \textbf{local analytic diffeomorphism} at the zero element in
$\mathfrak{h.}$

From here on, we fix a set $J$ as described above. Generally, the set
$\mathcal{O}$ needs not be relatively compact. However, it is always possible
to select a sufficiently small connected open and relatively compact subset
$\mathcal{O}_{\circ}\subset\mathcal{O}$ around the identity in $H$ such that

(a) The restriction $\left.  \Theta_{\lambda}\right\vert _{\varphi\left(
\mathcal{O}_{\circ}\right)  }:\varphi\left(  \mathcal{O}_{\circ}\right)
\rightarrow\Theta_{\lambda}\left(  \varphi\left(  \mathcal{O}_{\circ}\right)
\right)  $ of $\Theta_{\lambda}$ to $\varphi\left(  \mathcal{O}_{\circ
}\right)  $ is a diffeomorphism.

(b) $\Theta_{\lambda}\left(  \varphi\left(  \mathcal{O}_{\circ}\right)
\right)  $ is relatively compact in $P^{\ast}\left(  \mathfrak{n}^{\ast
}\right)  .$

To avoid cluster of notation, for the remainder of this work, we simply
replace $\mathcal{O}_{\circ}$ with $\mathcal{O}$. The following diagram
summarizes the connection between the maps $\beta_{\lambda},P^{\ast},\varphi$
and $\left.  \Theta_{\lambda}\right\vert _{\mathcal{O}}$ defined above

\begin{center}%
\begin{equation}
\text{\includegraphics[scale=0.35]{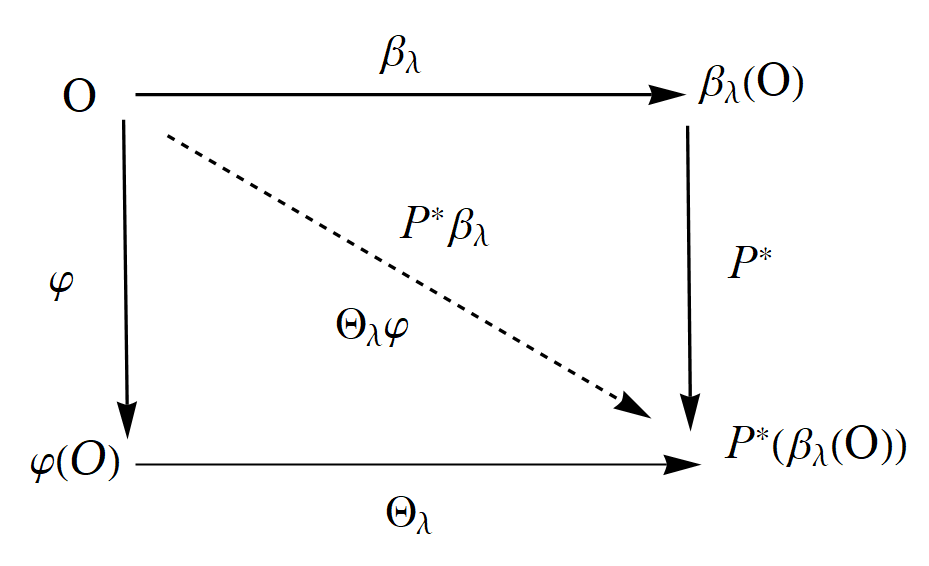}} \label{diagram}%
\end{equation}

\end{center}

\begin{proposition}
\label{discrete} The smooth function $h\mapsto Ad\left(  h^{-1}\right)
^{\ast}\lambda$ is an immersion at the identity of $H$ if and only if
$H_{\lambda}=\left\{  h\in H:Ad\left(  h^{-1}\right)  ^{\ast}\lambda
=\lambda\right\}  $ is a discrete subgroup of $H.$
\end{proposition}

\begin{proof}
The isotropy group $H_{\lambda}$ is a closed subgroup of $H$ and the map
$F:H/H_{\lambda}\rightarrow\beta_{\lambda}\left(  H\right)  $ defined by
$F\left(  hH_{\lambda}\right)  =Ad\left(  h^{-1}\right)  ^{\ast}\lambda$ is an
equivariant diffeomorphism (Theorem 7.19, \cite{lee2003smooth}.) As such,
$\dim\left(  H/H_{\lambda}\right)  =\dim\beta_{\lambda}\left(  H\right)  .$
Let us suppose that $h\mapsto Ad\left(  h^{-1}\right)  ^{\ast}\lambda$ is an
immersion at the identity of $H$ with rank $r.$ Then $r=\dim\left(
H/H_{\lambda}\right)  =\dim\beta_{\lambda}\left(  H\right)  =\dim H.$ This
implies that $H_{\lambda}$ is a discrete subgroup of $H$. Suppose on the other
hand that $H_{\lambda}$ is a discrete subgroup of $H$. Then there exists an
open set $O_{e}$ around the identity of $H$ such that $H_{\lambda}\cap O_{e}$
is trivial. Restricting $h\mapsto Ad\left(  h^{-1}\right)  ^{\ast}\lambda$ to
such an open set, it is then clear that $h\mapsto Ad\left(  h^{-1}\right)
^{\ast}\lambda$ must be an immersion at the identity of $H.$
\end{proof}

The purpose of the results presented below is to set the stage for the proofs
of our main theorems. Our primary objective is to prove that given a
continuous function $\mathbf{f}$ supported on $\mathcal{O}$, it is possible to
find a countable set $\Gamma\subset G$ such that $\left\{  \pi\left(
\gamma\right)  \mathbf{f}:\gamma\in\Gamma\right\}  $ is a frame for
$L^{2}\left(  H,d\mu_{H}\right)  .$ To this end, we will need to show that
there exist positive constants $A,B$ such that%
\[
A\left\Vert \mathbf{g}\right\Vert ^{2}\leq\sum_{\gamma\in\Gamma}\left\vert
\left\langle \mathbf{g,}\pi\left(  \gamma\right)  \mathbf{f}\right\rangle
\right\vert ^{2}\leq B\left\Vert \mathbf{g}\right\Vert ^{2}%
\]
for all $\mathbf{g\in}L^{2}\left(  H,d\mu_{H}\right)  .$ We will be mainly
interested in countable sets of the type $\Gamma=\Gamma_{H}\Gamma_{N}$ where
$\Gamma_{H}$ and $\Gamma_{N}$ are some discrete subsets of $H$ and
$\exp\left(  P\mathfrak{n}\right)  \subseteq N$ respectively. Since we are
committed to providing a unified discretization scheme for the class of
representations described above, the following observation is in order.

\begin{remark}
Generally, it is important to choose our sampling set to be of the form
$\Gamma_{H}\Gamma_{N}$ as opposed to $\Gamma_{N}\Gamma_{H}$ (which may at
first appear to be a more natural ordering). To see this, let us consider the
case of the ax+b Lie group which we realize as follows. Let
\[
G=NH=\left\{  \left[
\begin{array}
[c]{cc}%
e^{t} & x\\
0 & 1
\end{array}
\right]  :t,x\in%
\mathbb{R}
\right\}
\]
be the ax+b group with Lie algebra $\mathfrak{g}$ spanned by
\[
X=\left[
\begin{array}
[c]{cc}%
0 & 1\\
0 & 0
\end{array}
\right]  ,A=\left[
\begin{array}
[c]{cc}%
1 & 0\\
0 & 0
\end{array}
\right]  .
\]
The only non-trivial Lie bracket of $\mathfrak{g}$ is given by $\left[
A,X\right]  =X.$ Next, let $\lambda=X_{1}^{\ast}$. We consider the induced
representation $\pi_{\lambda}=\mathrm{ind}_{N}^{NH}\left(  \chi_{\lambda
}\right)  $ which we realize as acting in the Hilbert space $L^{2}\left(
\mathbb{R}
\right)  =L^{2}\left(  H\right)  $ as follows%
\[
\pi_{\lambda}\left(  \exp xX\right)  \mathbf{f}\left(  t\right)  =e^{2\pi
ie^{-t}x}\cdot\mathbf{f}\left(  t\right)  \text{ and }\pi_{\lambda}\left(
\exp aA\right)  \mathbf{f}\left(  t\right)  =\mathbf{f}\left(  t-a\right)  .
\]
For some fixed positive real number $\epsilon,$ let $I=\left(  -\epsilon
,\epsilon\right)  $ be an interval in $%
\mathbb{R}
.$ It is shown in \cite{oussa2018frames} that there exist a positive and a
continuous function $\mathbf{f}\in L^{2}\left(
\mathbb{R}
\right)  $ supported on $I$ and some positive real numbers $\alpha,\delta$
such that $\left\{  \pi_{\lambda}\left(  \exp\kappa A\right)  \pi_{\lambda
}\left(  \exp jX\right)  \mathbf{f}:j\in\alpha%
\mathbb{Z}
,\kappa\in\delta%
\mathbb{Z}
\right\}  $ is a frame for $L^{2}\left(
\mathbb{R}
\right)  .$ Next, we consider the following system
\[
\left\{  \pi_{\lambda}\left(  \exp jX\right)  \pi_{\lambda}\left(  \exp\kappa
A\right)  \mathbf{f}:j\in\alpha%
\mathbb{Z}
,\kappa\in\delta%
\mathbb{Z}
\right\}
\]
obtained by changing the order of the original sampling set. For
$\mathbf{h\in}L^{2}\left(
\mathbb{R}
\right)  ,$ we have
\[%
{\displaystyle\sum\limits_{j\in\alpha\mathbb{Z}}}
{\displaystyle\sum\limits_{\kappa\in\delta\mathbb{Z}}}
\left\vert \left\langle \mathbf{h,}\pi_{\lambda}\left(  \exp jX\right)
\pi_{\lambda}\left(  \exp\kappa A\right)  \mathbf{f}\right\rangle \right\vert
^{2}=%
{\displaystyle\sum\limits_{j\in\alpha\mathbb{Z}}}
{\displaystyle\sum\limits_{\kappa\in\delta\mathbb{Z}}}
\left\vert \int_{\kappa+I}\mathbf{h}\left(  t\right)  e^{-2\pi ie^{-t}j}%
\cdot\overline{\mathbf{f}\left(  t-\kappa\right)  }dt\right\vert ^{2}.
\]
The change of variable $\xi=e^{-t}$ gives
\begin{align*}
&
{\displaystyle\sum\limits_{j\in\alpha\mathbb{Z}}}
{\displaystyle\sum\limits_{\kappa\in\delta\mathbb{Z}}}
\left\vert \left\langle \mathbf{h,}\pi_{\lambda}\left(  \exp jX\right)
\pi_{\lambda}\left(  \exp\kappa A\right)  \mathbf{f}\right\rangle \right\vert
^{2}\\
&  =%
{\displaystyle\sum\limits_{j\in\alpha\mathbb{Z}}}
{\displaystyle\sum\limits_{\kappa\in\delta\mathbb{Z}}}
\left\vert \int_{e^{\kappa+I}}\frac{\mathbf{h}\left(  \log\left(  \xi
^{-1}\right)  \right)  \cdot\overline{\mathbf{f}\left(  \log\left(  \xi
^{-1}\right)  -\kappa\right)  }}{\xi}e^{-2\pi i\xi j}d\xi\right\vert ^{2}.
\end{align*}
Moreover, suppose that $\mathbf{h}$ is a nonzero function which is compactly
supported and additionally satisfies the following conditions: (1)
\begin{align*}
&
{\displaystyle\sum\limits_{j\in\alpha\mathbb{Z}}}
{\displaystyle\sum\limits_{\kappa\in\delta\mathbb{Z}}}
\left\vert \int_{e^{\kappa+I}}\frac{\mathbf{h}\left(  \log\left(  \xi
^{-1}\right)  \right)  \cdot\overline{\mathbf{f}\left(  \log\left(  \xi
^{-1}\right)  -\kappa\right)  }}{\xi}e^{-2\pi i\xi j}d\xi\right\vert ^{2}\\
&  =%
{\displaystyle\sum\limits_{j\in\alpha\mathbb{Z}}}
\left\vert \int_{e^{\kappa_{0}+I}}\frac{\mathbf{h}\left(  \log\left(  \xi
^{-1}\right)  \right)  \cdot\overline{\mathbf{f}\left(  \log\left(  \xi
^{-1}\right)  -\kappa_{0}\right)  }}{\xi}e^{-2\pi i\xi j}d\xi\right\vert ^{2}%
\end{align*}
for some fixed element $\kappa_{0}\in\delta\mathbb{Z}.$ (2) The Lebesgue
measure of the set $e^{\kappa_{0}+I}$ is strictly larger than $\alpha^{-1}.$
Note that the second condition is due to the expansive property of the
exponential function $x\mapsto e^{x}$ on the real line. As such, the
collection of trigonometric exponentials $\left\{  e^{-2\pi i\xi j}:j\in\alpha%
\mathbb{Z}
\right\}  $ cannot be total in $L^{2}\left(  e^{\kappa_{0}+I},d\xi\right)  $.
It is then possible to choose a nonzero function $\mathbf{h}\in L^{2}\left(
\mathbb{R}
\right)  $ such that
\[
\int_{e^{\kappa_{0}+I}}\frac{\mathbf{h}\left(  \log\left(  \xi^{-1}\right)
\right)  \cdot\overline{\mathbf{f}\left(  \log\left(  \xi^{-1}\right)
-\kappa_{0}\right)  }}{\xi}e^{-2\pi i\xi j}d\xi=0
\]
for every integer $j.$ Fixing such a function $\mathbf{h,}$ we obtain
\[%
{\displaystyle\sum\limits_{j\in\mathbb{Z}}}
{\displaystyle\sum\limits_{\kappa\in\mathbb{Z}}}
\left\vert \left\langle \mathbf{h,}\pi_{\lambda}\left(  \exp jX\right)
\pi_{\lambda}\left(  \exp\kappa A\right)  \mathbf{f}\right\rangle \right\vert
^{2}=0.
\]
Since $\mathbf{h}$ is not the zero function, we conclude that $\left\{
\pi_{\lambda}\left(  \exp jX\right)  \pi_{\lambda}\left(  \exp\kappa A\right)
\mathbf{f}:j\in\alpha%
\mathbb{Z}
,\kappa\in\delta%
\mathbb{Z}
\right\}  $ is not a frame for $L^{2}\left(
\mathbb{R}
\right)  .$
\end{remark}

Given $\mathbf{g}$ and $\mathbf{f}$ in $L^{2}\left(  H,d\mu_{H}\right)  ,$ we
formally define
\begin{equation}
\mathrm{s}\left(  \Gamma_{H},\Gamma_{N},\mathbf{g,f}\right)  =\sum_{\ell
\in\Gamma_{H}}\sum_{\exp PX\in\Gamma_{N}}\left\vert \left\langle
\mathbf{g},\pi\left(  \ell^{-1}\right)  \pi\left(  \exp\left(  PX\right)
\right)  \mathbf{f}\right\rangle _{\mathfrak{H}}\right\vert ^{2}.
\end{equation}
At this point, it is not clear for which quadruple $\left(  \Gamma_{H}%
,\Gamma_{N},\mathbf{g,f}\right)  $ should the series defined by $\mathrm{s}%
\left(  \Gamma_{H},\Gamma_{N},\mathbf{g,f}\right)  $ be convergent (or
divergent). Suppose that $\mathbf{f}$ is continuous and supported on
$\mathcal{O}$ and for the time being, we shall overlook all convergence issues
related to the definition of $\mathrm{s}\left(  \Gamma_{H},\Gamma
_{N},\mathbf{g,f}\right)  .$ Since
\[
\left[  \pi\left(  \exp\left(  X\right)  \right)  \right]  \mathbf{f}\left(
h\right)  =e^{2\pi i\left\langle Ad\left(  h^{-1}\right)  ^{\ast}%
\lambda,X\right\rangle }\mathbf{f}\left(  h\right)  \text{ and }\left[
\pi\left(  \ell\right)  \right]  \mathbf{f}\left(  h\right)  =\mathbf{f}%
\left(  \ell^{-1}h\right)  ,
\]
it is clear that in a formal sense, we have
\[
I_{\mathbf{g,f,}\ell,X}=\left\langle \mathbf{g},\pi\left(  \ell^{-1}\right)
\pi\left(  \exp\left(  PX\right)  \right)  \mathbf{f}\right\rangle
_{\mathfrak{H}}=\int_{\mathcal{O}}\mathbf{g}\left(  \ell^{-1}h\right)
e^{-2\pi i\left\langle Ad\left(  h^{-1}\right)  ^{\ast}\lambda,PX\right\rangle
}\overline{\mathbf{f}\left(  h\right)  }d\mu_{H}\left(  h\right)  .
\]
Recall that $\beta_{\lambda}:H\rightarrow\beta_{\lambda}\left(  H\right)
\subseteq\mathfrak{n}^{\ast}$ is defined as follows: $\ \beta_{\lambda}\left(
h\right)  =Ad\left(  h^{-1}\right)  ^{\ast}\lambda.$ Next, for each $h\in$
$\mathcal{O}$, there exists a unique element $A$ in $\varphi\left(
\mathcal{O}\right)  $ (which is a subset of the Lie algebra of $H$) such that
$h=\varphi^{-1}\left(  A\right)  $ and
\[
I_{\mathbf{g,f,}\ell,X}=\int_{\varphi\left(  \mathcal{O}\right)  }\left(
\mathbf{g}\left(  \ell^{-1}\varphi^{-1}\left(  A\right)  \right)  e^{-2\pi
i\left\langle \beta_{\lambda}\left(  \varphi^{-1}\left(  A\right)  \right)
,PX\right\rangle }\overline{\mathbf{f}\left(  \varphi^{-1}\left(  A\right)
\right)  }\right)  d\mu_{H}\left(  \varphi^{-1}\left(  A\right)  \right)  .
\]
Since $P$ is a linear projection, we obtain
\[
\left\langle \beta_{\lambda}\left(  \varphi^{-1}\left(  A\right)  \right)
,PX\right\rangle =\left\langle P^{\ast}\beta_{\lambda}\left(  \varphi
^{-1}\left(  A\right)  \right)  ,PX\right\rangle ;
\]
and as such,
\[
I_{\mathbf{g,f,}\ell,X}=\int_{\varphi\left(  \mathcal{O}\right)  }\left(
\mathbf{g}\left(  \ell^{-1}\varphi^{-1}\left(  A\right)  \right)
\overline{\mathbf{f}\left(  \varphi^{-1}\left(  A\right)  \right)  }\right)
e^{-2\pi i\left\langle P^{\ast}\beta_{\lambda}\left(  \varphi^{-1}\left(
A\right)  \right)  ,PX\right\rangle }d\mu_{H}\left(  \varphi^{-1}\left(
A\right)  \right)  .
\]
In local coordinates, the action of $P^{\ast}\beta_{\lambda}$ is given by the
diffeomorphism
\[
\Theta_{\lambda}=P^{\ast}\beta_{\lambda}\varphi^{-1}:\varphi\left(
\mathcal{O}\right)  \rightarrow\Theta_{\lambda}\left(  \varphi\left(
\mathcal{O}\right)  \right)
\]
and for every $A\in\varphi\left(  \mathcal{O}\right)  ,$ there exists a unique
element $\xi$ in $\Theta_{\lambda}\left(  \mathcal{O}\right)  \subset
\mathfrak{p}^{\ast}$ such that, given $A=\Theta_{\lambda}^{-1}\left(
\xi\right)  ,$ and
\[
F_{\mathbf{g,f}}\left(  \xi\right)  =\mathbf{g}\left(  \ell^{-1}\varphi
^{-1}\left(  \Theta_{\lambda}^{-1}\left(  \xi\right)  \right)  \right)
\overline{\mathbf{f}\left(  \varphi^{-1}\left(  \Theta_{\lambda}^{-1}\left(
\xi\right)  \right)  \right)  }%
\]
we have%
\[
I_{\mathbf{g,f,}\ell,X}=\int_{\Theta_{\lambda}\left(  \varphi\left(
\mathcal{O}\right)  \right)  }F_{\mathbf{g,f}}\left(  \xi\right)  e^{-2\pi
i\left\langle \xi,PX\right\rangle }d\mu_{H}\left(  \varphi^{-1}\left(
\Theta_{\lambda}^{-1}\left(  \xi\right)  \right)  \right)  .
\]
Note that $d\mu_{H}\left(  \varphi^{-1}\left(  \Theta_{\lambda}^{-1}\left(
\xi\right)  \right)  \right)  $ is a measure on $\Theta_{\lambda}\left(
\varphi\left(  \mathcal{O}\right)  \right)  .$ In fact, $\Theta_{\lambda
}\left(  \varphi\left(  \mathcal{O}\right)  \right)  $ is merely the
pushforward of the Haar measure on $H$ under the map $\Theta_{\lambda}%
\circ\varphi.$

Let $d\xi$ be the Lebesgue measure on $P^{\ast}\left(  \mathfrak{n}^{\ast
}\right)  .$ To simplify the integral $I_{\mathbf{g,f,}\ell,X}$, we will need
to compute the Radon-Nikodym derivative $\frac{d\mu_{H}\left(  \varphi
^{-1}\left(  \Theta_{\lambda}^{-1}\left(  \xi\right)  \right)  \right)  }%
{d\xi}.$ To this end, we introduce the following auxiliary function. Define
$\mathbf{W}_{\lambda}:\Theta_{\lambda}\left(  \varphi\left(  \mathcal{O}%
\right)  \right)  \rightarrow\left(  0,\infty\right)  $ such that
\begin{equation}
\mathbf{W}_{\lambda}\left(  \xi\right)  =\frac{d\mu_{H}\left(  \varphi
^{-1}\left(  \Theta_{\lambda}^{-1}\left(  \xi\right)  \right)  \right)  }%
{d\xi}. \label{W}%
\end{equation}
In light of the discussion above, the following is immediate.

\begin{lemma}
\label{LemmaS}Given discrete sets $\Gamma_{H}\subset H$, $\Gamma_{N}%
\subset\exp\left(  P\mathfrak{n}\right)  $, $\mathbf{g}$ and $\mathbf{f}$ in
$L^{2}\left(  H,d\mu_{H}\right)  $ such that $\mathbf{f}$ is a continuous
function supported in $\mathcal{O},$
\begin{align*}
&  \mathrm{s}\left(  \Gamma_{H},\Gamma_{N},\mathbf{g,f}\right) \\
&  =\sum_{\ell\in\Gamma_{H}}\sum_{\exp\left(  PX\right)  \in\Gamma_{N}%
}\left\vert \int_{\Theta_{\lambda}\left(  \varphi\left(  \mathcal{O}\right)
\right)  }\left[  \mathbf{g}\left(  \ell^{-1}\varphi^{-1}\left(
\Theta_{\lambda}^{-1}\left(  \xi\right)  \right)  \right)  \overline
{\mathbf{f}\left(  \varphi^{-1}\left(  \Theta_{\lambda}^{-1}\left(
\xi\right)  \right)  \right)  }\mathbf{W}_{\lambda}\left(  \xi\right)
\right]  e^{-2\pi i\left\langle \xi,PX\right\rangle }d\xi\right\vert ^{2}.
\end{align*}

\end{lemma}

\begin{lemma}
\label{Jacob}$\mathbf{W}_{\lambda}$ is a strictly positive function defined on
$\Theta_{\lambda}\left(  \varphi\left(  \mathcal{O}\right)  \right)  $. More
precisely,
\[
\mathbf{W}_{\lambda}\left(  \xi\right)  =\left\vert
{\displaystyle\prod\limits_{k=1}^{r}}
v_{k}\left(  \xi\right)  \right\vert \cdot\left\vert \det\left[
\dfrac{\partial\left[  \Theta_{\lambda}^{-1}\left(  \xi\right)  \right]  _{j}%
}{\partial\xi_{k}}\right]  _{1\leq j,k\leq n}\right\vert
\]
where for $k\in\left\{  1,2,\cdots,r\right\}  ,$%
\[
v_{k}\left(  \xi\right)  =\left\{
\begin{array}
[c]{c}%
\dfrac{e^{w_{k}\left(  \xi\right)  }-1}{w_{k}\left(  \xi\right)  }\text{ if
}w_{k}\left(  \xi\right)  \neq0\\
1\text{ if }w_{k}\left(  \xi\right)  =0
\end{array}
\right.
\]
and $\left(  w_{k}\left(  \xi\right)  \right)  _{k=1}^{r}$ is a sequence of
eigenvalues for the endomorphism $-ad_{\mathfrak{h}}\left(  \Theta_{\lambda
}^{-1}\left(  \xi\right)  \right)  .$
\end{lemma}

\begin{proof}
Referring back to the diagram given in Figure \ref{diagram}, we proceed as
follows. First, let $\mathbf{f}$ be a positive function compactly supported on
$\mathcal{O}$. Given $h\in\mathcal{O},$ since there exists a unique
$A\in\varphi\left(  \mathcal{O}\right)  \subset\mathfrak{h}$ such that
$h=\varphi^{-1}\left(  A\right)  ,$ we have
\[
\int_{\mathcal{O}}\mathbf{f}\left(  h\right)  d\mu_{H}\left(  h\right)
=\int_{\varphi\left(  \mathcal{O}\right)  }\mathbf{f}\left(  \varphi
^{-1}\left(  A\right)  \right)  \cdot d\mu_{H}\left(  \varphi^{-1}\left(
A\right)  \right)  .
\]
Next,
\[
\int_{\mathcal{O}}\mathbf{f}\left(  h\right)  d\mu_{H}\left(  h\right)
=\int_{\Theta_{\lambda}\left(  \varphi\left(  \mathcal{O}\right)  \right)
}\mathbf{f}\left(  \varphi^{-1}\left(  \Theta_{\lambda}^{-1}\left(
\xi\right)  \right)  \right)  d\mu_{H}\left(  \varphi^{-1}\left(
\Theta_{\lambda}^{-1}\left(  \xi\right)  \right)  \right)  .
\]
Taking $\varphi$ to be the logarithmic map (we assume here that $\mathcal{O}$
is sufficiently small), we obtain
\[
\int_{\mathcal{O}}\mathbf{f}\left(  h\right)  d\mu_{H}\left(  h\right)
=\int_{\Theta_{\lambda}\left(  \varphi\left(  \mathcal{O}\right)  \right)
}\mathbf{f}\left(  \exp\left(  \Theta_{\lambda}^{-1}\left(  \xi\right)
\right)  \right)  d\mu_{H}\left(  \exp\left(  \Theta_{\lambda}^{-1}\left(
\xi\right)  \right)  \right)  .
\]
Chain Rule together with \cite[Proposition 5.5.6]{faraut2008analysis} gives
the following. Defining
\[
A\left(  \xi\right)  =\frac{\mathrm{id}-\exp\left(  -ad_{\mathfrak{h}}\left(
\Theta_{\lambda}^{-1}\left(  \xi\right)  \right)  \right)  }{ad_{\mathfrak{h}%
}\left(  \Theta_{\lambda}^{-1}\left(  \xi\right)  \right)  };
\]
up to multiplication by a positive constant, we have
\begin{equation}
\mathbf{W}_{\lambda}\left(  \xi\right)  =\left\vert \det\left(  A\left(
\xi\right)  \right)  \cdot\det\left[  \dfrac{\partial\left[  \Theta_{\lambda
}^{-1}\left(  \xi\right)  \right]  _{j}}{\partial\xi_{k}}\right]  _{1\leq
j,k\leq n}\right\vert .
\end{equation}
Note that the function $\xi\mapsto\det A\left(  \xi\right)  $ is nonzero on
$\Theta_{\lambda}\left(  \varphi\left(  \mathcal{O}\right)  \right)  $ since
each linear operator
\begin{equation}
A\left(  \xi\right)  =\sum_{k=0}^{\infty}\frac{\left(  -ad_{\mathfrak{h}%
}\left(  \Theta_{\lambda}^{-1}\left(  \xi\right)  \right)  \right)  ^{k}%
}{\left(  k+1\right)  !}%
\end{equation}
is invertible \cite[Page 25]{duistermaat2012lie}. Precisely, there is an
ordering $\left(  v_{1}\left(  \xi\right)  ,\cdots,v_{r}\left(  \xi\right)
\right)  $ for the eigenvalues of the linear operator $\frac{\mathrm{id}%
-\exp\left(  -ad_{\mathfrak{h}}\left(  \Theta_{\lambda}^{-1}\left(
\xi\right)  \right)  \right)  }{ad_{\mathfrak{h}}\left(  \Theta_{\lambda}%
^{-1}\left(  \xi\right)  \right)  }$ as well as an ordering $\left(
w_{1}\left(  \xi\right)  ,\cdots,w_{r}\left(  \xi\right)  \right)  $ for the
eigenvalues of $-ad_{\mathfrak{h}}\left(  \Theta_{\lambda}^{-1}\left(
\xi\right)  \right)  $ such that
\begin{equation}
\left\vert \det\left(  A\left(  \xi\right)  \right)  \right\vert =\left\vert
{\displaystyle\prod\limits_{k=1}^{r}}
v_{k}\left(  \xi\right)  \right\vert
\end{equation}
and for each index $k,$
\[
v_{k}\left(  \xi\right)  =\left\{
\begin{array}
[c]{c}%
\dfrac{e^{w_{k}\left(  \xi\right)  }-1}{w_{k}\left(  \xi\right)  }\text{ if
}w_{k}\left(  \xi\right)  \neq0\\
1\text{ if }w_{k}\left(  \xi\right)  =0
\end{array}
\right.  .
\]
As a result,
\begin{equation}
\mathbf{W}_{\lambda}\left(  \xi\right)  =\left\vert
{\displaystyle\prod\limits_{k=1}^{r}}
v_{k}\left(  \xi\right)  \right\vert \cdot\left\vert \det\left(  \left[
\dfrac{\partial\left[  \Theta_{\lambda}^{-1}\left(  \xi\right)  \right]  _{j}%
}{\partial\xi_{k}}\right]  _{1\leq j,k\leq n}\right)  \right\vert
\end{equation}
and for $\xi\in\Theta_{\lambda}\left(  \varphi\left(  \mathcal{O}\right)
\right)  ,$%
\[
\mathbf{W}_{\lambda}\left(  \xi\right)  >0.
\]

\end{proof}

\begin{lemma}
For any continuous function $\mathbf{s}$ compactly supported on $\mathcal{O},$
the function
\[
h\mapsto\mathbf{s}\left(  h\right)  \sqrt{\mathbf{W}_{\lambda}\left(
\Theta_{\lambda}\left(  \varphi\left(  h\right)  \right)  \right)  }%
\]
is also continuous and supported on $\mathcal{O}$.
\end{lemma}

\begin{proof}
According to Lemma \ref{Jacob}, $\mathbf{W}_{\lambda}$ is a strictly positive
function defined on $\Theta_{\lambda}\left(  \varphi\left(  \mathcal{O}%
\right)  \right)  .$ Thus, the function $h\mapsto\sqrt{\mathbf{W}_{\lambda
}\left(  \Theta_{\lambda}\left(  \varphi\left(  h\right)  \right)  \right)  }$
is a positive continuous function on $\mathcal{O}$. Furthermore, for any
continuous function $\mathbf{s}$ supported on $\mathcal{O}$, it is clear that
\[
h\mapsto\mathbf{s}\left(  h\right)  \cdot\sqrt{\mathbf{W}_{\lambda}\left(
\Theta_{\lambda}\left(  \varphi\left(  h\right)  \right)  \right)  }%
\]
is also continuous and supported on $\mathcal{O}$.
\end{proof}

\begin{remark}
Although Lemma \ref{Jacob} is useful for the proofs of the main results, it is
not practical for specific examples (see the examples given in Section
\ref{W}). To compute $\mathbf{W}_{\lambda}$ it is clearly important to have at
our disposal, an explicit formula for the Haar measure of $H$. In many cases,
this can be achieved inductively as follows. Suppose that $H=ST$ where $S,T$
are closed subgroups, $S\cap T$ is compact and $S\times T\backepsilon\left(
s,t\right)  \mapsto st\in H$ is a homeomorphism. Then
\[
\int_{H}\mathbf{f}\left(  h\right)  d\mu_{H}\left(  h\right)  =\int_{S}%
\int_{T}\mathbf{f}\left(  st\right)  \frac{\det\left(  Ad_{T}\left(  t\right)
\right)  }{\det\left(  Ad_{S}\left(  s\right)  \right)  }d\mu_{T}\left(
t\right)  d\mu_{S}\left(  s\right)
\]
and $d\mu_{S}\left(  s\right)  d\mu_{T}\left(  t\right)  $ are left Haar
measure for $S$ and $T$ respectively \cite[Proposition 5.26]%
{humphreys2012introduction}.
\end{remark}

In our quest to describe quadruples $\left(  \Gamma_{H},\Gamma_{N}%
,\mathbf{g,f}\right)  $ for which the series $\mathrm{s}\left(  \Gamma
_{H},\Gamma_{N},\mathbf{g,f}\right)  $ is convergent, we will need to also
address the following question: Let $\mathbf{f}$ be a nonzero positive
function which is compactly supported on $H.$ Can we find a discrete set
$\Lambda\subset H$ such that for all $x\in H,$ there exist positive real
numbers $A\leq B$\ not depending on $x$ such that $A\leq\sum_{\gamma\in
\Lambda}\mathbf{f}\left(  \gamma x\right)  \leq B?$ A positive answer to this
question is given by Lemma \ref{communicated}; and its proof \footnote{Many
thanks go to H. F\"{u}hr for sharing with us a complete proof of Lemma
\ref{communicated}.} requires the following concepts.

\begin{itemize}
\item A discrete set $\Gamma\subset H$ is $U$-separated if $\gamma U\cap
\gamma^{\prime}U$ is empty for $\gamma\neq\gamma^{\prime}$ in $\Gamma.$

\item A discrete set $\Gamma\subset H$ is called $V$-dense if $H=\Gamma V.$
\end{itemize}

\begin{lemma}
\label{communicated} Let $\mathbf{f}$ be a nonzero positive function which is
compactly supported on $H.$ Then there exists a discrete set $\Gamma
_{H}\subset H$ such that the compound inequalities%
\[
A\leq\sum_{\gamma\in\Gamma_{H}}\mathbf{f}\left(  \gamma x\right)  \leq B
\]
hold for all $x\in H$ with $0<A\leq B<\infty$ independent of $x.$
\end{lemma}

\begin{proof}
Fixing a set $\Gamma\subset H$ which is $U$-separated, for some open set $U$
around the neutral element of $H$, we select an open symmetric set $V$ such
that $V^{2}\subset U.$ Next, let
\[
\mathrm{supp}\left(  f\right)  =%
{\displaystyle\bigcup\limits_{i=1}^{k}}
x_{i}V\text{ where }x_{i}\in\Gamma.
\]
Then each set $x_{i}V$ contains no more than one element from $\Gamma.$
Suppose otherwise. Then given $\kappa,\gamma\in\Gamma$ in $x_{i}V$ one has
that $\kappa^{-1}\gamma\in V^{2}\subset U$ and this contradicts the fact that
$\Gamma$ is $U$-separated. Thus,
\[
\sum_{\gamma\in\Gamma}\mathbf{f}\left(  \gamma x\right)  \leq k\cdot\left(
\sup\left\{  f\left(  x\right)  :x\in H\right\}  \right)  <\infty.
\]
Next, let $W$ be an open set contained in $H$ such that $\mathbf{f}\left(
x\right)  >\epsilon$ for all $x\in W$ and for some positive real number
$\epsilon>0.$ Without loss of generality, we may assume that $W$ is an open
set around the neutral element in $H.$ Let us select a symmetric open set $V$
around the identity in $H$ such that $V^{2}\subset W.$ Next, observe that
every $V$-dense subset $\Gamma$ must intersect $W.$ To see this, let $y\in V$
and $\gamma\in\Gamma$ such that $y\in\gamma V.$ Then clearly $\gamma\in
yV^{-1}\subset V^{2}\subset W$ and
\[
\sum_{\gamma\in\Gamma}\mathbf{f}\left(  \gamma x\right)  \geq\epsilon>0.
\]
To complete the proof, it is enough to select a $V$-dense, $U$-separated set
with sufficiently small $V.$ To this end, pick $Z$ contained in $V$ such that
$Z$ is symmetric, $Z^{2}$ is contained in $V$ and pick a $Z$-separated subset
which is maximal with respect to inclusion. Then such a set is $V$-dense and
$Z$-separated.
\end{proof}

\begin{remark}
In the case where $H$ is completely solvable, an explicit construction of the
set $\Gamma_{H}$ in Lemma \ref{communicated} can be found in \cite[Lemma 29]{oussa2018frames}.
\end{remark}

Let
\begin{equation}
\mathbf{w}\left(  h\right)  =\mathbf{W}_{\lambda}\left(  \Theta_{\lambda
}\left(  \varphi\left(  h\right)  \right)  \right)  .
\end{equation}

\begin{remark}
\label{mM} Given a continuous (or a smooth) function $\mathbf{f}$ which is
supported inside $\mathcal{O}$, since the product $\mathbf{f}\sqrt{\mathbf{w}%
}$ is continuous and supported inside $\mathcal{O}$, according to Lemma
\ref{communicated}, there exists a discrete subset $\Gamma_{H}$ of $H$ such
that given
\begin{equation}
m_{H,\mathbf{f}}=\mathrm{essinf}_{h\in H}\left(  \sum_{\ell\in\Gamma_{H}%
}\left\vert \mathbf{f}\left(  \ell h\right)  \cdot\sqrt{\mathbf{w}\left(  \ell
h\right)  }\right\vert ^{2}\right)  \text{ and }M_{H,\mathbf{f}}%
=\mathrm{esssup}_{h\in H}\left(  \sum_{\ell\in\Gamma_{H}}\left\vert
\mathbf{f}\left(  \ell h\right)  \cdot\sqrt{\mathbf{w}\left(  \ell h\right)
}\right\vert ^{2}\right)  , \label{mM}%
\end{equation}
we have
\[
0<m_{H,\mathbf{f}}\leq M_{H,\mathbf{f}}<\infty.
\]

\end{remark}

We recall that by assumption $\Theta_{\lambda}\left(  \varphi\left(
\mathcal{O}\right)  \right)  $ is relatively compact in $P^{\ast}\left(
\mathfrak{n}^{\ast}\right)  $ and $\Theta_{\lambda}\left(  \varphi\left(
\mathcal{O}\right)  \right)  $ has positive Lebesgue measure in
\begin{equation}
\mathfrak{p}^{\ast}=P^{\ast}\left(  \mathfrak{n}^{\ast}\right)  \simeq
\mathbb{R}^{r}.\label{pstar}%
\end{equation}
Given $C\subset\mathfrak{p}^{\ast},$ and $\Gamma_{N}$ contained in
$\exp(\mathfrak{p}),$ we define
\begin{equation}
\mathcal{E}\left(  C,\Gamma_{N}\right)  =\left\{  \frac{e^{2\pi i\left\langle
\xi,Y\right\rangle }}{\left\vert C\right\vert ^{1/2}}:Y\in\log\left(
\Gamma_{N}\right)  \right\}  .\label{EC}%
\end{equation}
Put
\begin{equation}
\mathfrak{p}=P_{J}\left(  \mathfrak{n}\right)  =P\left(  \mathfrak{n}\right)
.\label{p}%
\end{equation}

\begin{lemma}
\label{tstar}There exists a set $C$ which is a compact subset of
$\mathfrak{p}^{\ast}$ such that $\Theta_{\lambda}\left(  \varphi\left(
\mathcal{O}\right)  \right)  \subseteq C$ and $\Gamma_{N}$ is a discrete
subset of $\exp\left(  \mathfrak{p}\right)  $ such that the trigonometric
system $\mathcal{E}\left(  C,\Gamma_{N}\right)  $ is an orthonormal basis for
$L^{2}\left(  C,d\xi\right)  .$
\end{lemma}

\begin{proof}
Define a positive real number $\varsigma$ such that
\[
\varsigma=\mathrm{esssup}_{h\in H}\left\{  \left\Vert \Theta_{\lambda}\left(
\varphi\left(  h\right)  \right)  \right\Vert _{\max}:h\in\mathcal{O}\right\}
.
\]
Next, let $J=\left\{  j_{1}<\cdots<j_{r}\right\}  =\left\{  1,2,\cdots
,n\right\}  .$ Put $C=\sum_{k=1}^{r}\left[  -\varsigma,\varsigma\right)
X_{j_{k}}^{\ast}\subset\mathfrak{p}^{\ast}$ and define
\[
\Gamma_{N}=\exp\left(  \sum_{k=1}^{r}\frac{1}{2\varsigma}%
\mathbb{Z}
X_{j_{k}}\right)  \subset N.
\]
Then $\left\{  C+k:k\in\sum_{k=1}^{r}2\varsigma%
\mathbb{Z}
X_{j_{k}}^{\ast}\right\}  $ tiles $\mathfrak{p}^{\ast}.$ Moreover, the
Lebesgue measure of $C$ is equal $\left(  2\varsigma\right)  ^{r}$ and the
system $\mathcal{E}\left(  C,\Gamma_{N}\right)  $ is an orthonormal basis for
$L^{2}\left(  C,d\xi\right)  .$
\end{proof}

To simplify our presentation, we list below some important conditions which we
shall refer to throughout this work. \newpage

\begin{condition}
\label{Cond} \text{ }

\begin{enumerate}
\item $\Gamma_{H}$ is a discrete subset of $H$ such that
\begin{align*}
0  &  <m_{H,\mathbf{f}}=\mathrm{essinf}_{h\in H}\left(  \sum_{\ell\in
\Gamma_{H}}\left\vert \mathbf{f}\left(  \ell h\right)  \sqrt{\mathbf{w}\left(
\ell h\right)  }\right\vert ^{2}\right) \\
&  \leq M_{H,\mathbf{f}}=\mathrm{esssup}_{h\in H}\left(  \sum_{\ell\in
\Gamma_{H}}\left\vert \mathbf{f}\left(  \ell h\right)  \sqrt{\mathbf{w}\left(
\ell h\right)  }\right\vert ^{2}\right)  <\infty
\end{align*}
as defined in (\ref{mM}).

\item $C$ is a compact subset of $\mathfrak{p}^{\ast}$ such that
$\Theta_{\lambda}\left(  \varphi\left(  \mathcal{O}\right)  \right)  \subseteq
C$ and $\Gamma_{N}$ is a discrete subset of $\exp\left(  \mathfrak{p}\right)
$ such that the trigonometric system
\[
\mathcal{E}\left(  C,\Gamma_{N}\right)  =\left\{  \frac{e^{2\pi i\left\langle
\xi,Y\right\rangle }}{\left\vert C\right\vert ^{1/2}}:Y\in\log\left(
\Gamma_{N}\right)  \right\}
\]
is an orthonormal basis for $L^{2}\left(  C,d\xi\right)  $ as defined
(\ref{EC}).

\item $m_{H,\mathbf{f}}=M_{H,\mathbf{f}}$ (see (\ref{mM}))

\item $m_{H,\mathbf{f}}=M_{H,\mathbf{f}}=\left\vert C\right\vert ^{-1}$ (see
(\ref{mM}) and (\ref{EC}))

\item $\int_{H}\left\vert \mathbf{f}\left(  h\right)  \right\vert ^{2}d\mu
_{H}\left(  h\right)  =1$

\item There is a discrete subset $\Gamma_{H}$ of $H$ such that $\left\{
\ell^{-1}\mathcal{O}:\ell\in\Gamma_{H}\right\}  $ is a tiling of $H.$

\item $\mathbf{f}$ is a vector in $L^{2}\left(  \mathcal{O},d\mu_{H}\right)  $
such that
\[
\mathbf{f}\left(  \varphi^{-1}\left(  \Theta_{\lambda}^{-1}\left(  \xi\right)
\right)  \right)  =\mathbf{W}_{\lambda}\left(  \xi\right)  ^{-1/2}%
\cdot1_{\Theta_{\lambda}\left(  \varphi\left(  \mathcal{O}\right)  \right)
}\left(  \xi\right)
\]
and $\mathbf{W}_{\lambda}\left(  \xi\right)  =\frac{d\mu_{H}\left(
\varphi^{-1}\left(  \Theta_{\lambda}^{-1}\left(  \xi\right)  \right)  \right)
}{d\xi}.$

\item $\mathbf{f}$ is a vector in $L^{2}\left(  \mathcal{O},d\mu_{H}\right)  $
such that $\frac{\left\Vert \mathbf{f}\right\Vert _{\mathfrak{H}}}%
{\sqrt{\left\vert C\right\vert }}=1.$
\end{enumerate}
\end{condition}

For $\Gamma\subset G,$ we define
\begin{equation}
\mathcal{S}\left(  \mathbf{f,}\Gamma\right)  =\left\{  \pi\left(  x\right)
\mathbf{f}:x\in\Gamma\right\}  .
\end{equation}

\begin{lemma}
\label{FProperty} Let $\mathbf{f}$ be a continuous (or a smooth) function
which is compactly supported on $\mathcal{O}$. Next, let $\Gamma_{H}$ and
$\Gamma_{N}$ be discrete sets satisfying the conditions described in Condition
\ref{Cond} (1), and (2). Then $\mathcal{S}\left(  \mathbf{f,}\Gamma_{H}%
^{-1}\Gamma_{N}\right)  $ is a frame for $\mathfrak{H}$ with frame bounds
\[
0<m_{H,\mathbf{f}}\cdot\left\vert C\right\vert \leq M_{H,\mathbf{f}}%
\cdot\left\vert C\right\vert <\infty.
\]
Additionally, the following holds true.

\begin{enumerate}
\item If Condition \ref{Cond} (3) holds then $\mathcal{S}\left(
\mathbf{f,}\Gamma_{H}^{-1}\Gamma_{N}\right)  $ is a tight frame for
$\mathfrak{H}$.

\item If Condition \ref{Cond} (4) and (5) hold then $\mathcal{S}\left(
\mathbf{f,}\Gamma_{H}^{-1}\Gamma_{N}\right)  $ is an orthonormal basis for
$\mathfrak{H}.$
\end{enumerate}
\end{lemma}

\begin{proof}
Let $\mathbf{g}$ be a continuous function which is compactly supported in $H$.
Put
\[
\mathrm{s}\left(  \Gamma_{H},\Gamma_{N},\mathbf{g,f}\right)  =\sum_{\ell
\in\Gamma_{H}}\sum_{\exp\left(  PX\right)  \in\Gamma_{N}}\left\vert
\left\langle \mathbf{g},\pi\left(  \ell^{-1}\right)  \pi\left(  \exp\left(
PX\right)  \right)  \mathbf{f}\right\rangle _{\mathfrak{H}}\right\vert ^{2}.
\]
According to Lemma \ref{LemmaS},
\begin{align*}
&  \mathrm{s}\left(  \Gamma_{H},\Gamma_{N},\mathbf{g,f}\right) \\
&  =\sum_{\ell\in\Gamma_{H}}\sum_{\exp\left(  PX\right)  \in\Gamma_{N}%
}\left\vert \int_{\Theta_{\lambda}\left(  \varphi\left(  \mathcal{O}\right)
\right)  }\left[  \mathbf{g}\left(  \ell^{-1}\varphi^{-1}\left(
\Theta_{\lambda}^{-1}\left(  \xi\right)  \right)  \right)  \overline
{\mathbf{f}\left(  \varphi^{-1}\left(  \Theta_{\lambda}^{-1}\left(
\xi\right)  \right)  \right)  }\mathbf{W}_{\lambda}\left(  \xi\right)
\right]  e^{-2\pi i\left\langle \xi,PX\right\rangle }d\xi\right\vert ^{2}.
\end{align*}
To simplify notation, we let
\[
u\left(  \xi\right)  =\varphi^{-1}\left(  \Theta_{\lambda}^{-1}\left(
\xi\right)  \right)  \in H.
\]
Then%
\[
\mathrm{s}\left(  \Gamma_{H},\Gamma_{N},\mathbf{g,f}\right)  =\sum_{\ell
\in\Gamma_{H}}\sum_{\exp\left(  PX\right)  \in\Gamma_{N}}\left\vert
\int_{\Theta_{\lambda}\left(  \varphi\left(  \mathcal{O}\right)  \right)
}\mathbf{g}\left(  \ell^{-1}u\left(  \xi\right)  \right)  \overline
{\mathbf{f}\left(  u\left(  \xi\right)  \right)  }\mathbf{W}_{\lambda}\left(
\xi\right)  e^{-2\pi i\left\langle \xi,PX\right\rangle }d\xi\right\vert ^{2}.
\]
Letting $\mathfrak{F}_{C}$ be the Fourier series defined on $L^{2}(C),$ it is
clear that
\begin{align*}
&  \int_{\Theta_{\lambda}\left(  \varphi\left(  \mathcal{O}\right)  \right)
}\mathbf{g}\left(  \ell^{-1}u\left(  \xi\right)  \right)  \overline
{\mathbf{f}\left(  u\left(  \xi\right)  \right)  }\mathbf{W}_{\lambda}\left(
\xi\right)  e^{-2\pi i\left\langle \xi,PX\right\rangle }d\xi\\
&  =\mathfrak{F}_{C}\left(  \xi\mapsto\mathbf{g}\left(  \ell^{-1}u\left(
\xi\right)  \right)  \overline{\mathbf{f}\left(  u\left(  \xi\right)  \right)
}\mathbf{W}_{\lambda}\left(  \xi\right)  1_{\Theta_{\lambda}\left(
\varphi\left(  \mathcal{O}\right)  \right)  }\left(  \xi\right)  \right)
\left(  PX\right)  .
\end{align*}
Appealing to Plancherel's theorem,
\[
\mathrm{s}\left(  \Gamma_{H},\Gamma_{N},\mathbf{g,f}\right)  =\left\vert
C\right\vert \cdot\sum_{\ell\in\Gamma_{H}}\int_{\Theta_{\lambda}\left(
\varphi\left(  \mathcal{O}\right)  \right)  }\left\vert \mathbf{g}\left(
\ell^{-1}u\left(  \xi\right)  \right)  \overline{\mathbf{f}\left(  u\left(
\xi\right)  \right)  }\mathbf{W}_{\lambda}\left(  \xi\right)  ^{1/2}%
\right\vert ^{2}\mathbf{W}_{\lambda}\left(  \xi\right)  d\xi.
\]
Setting
\[
h=u\left(  \xi\right)  =\varphi^{-1}\left(  \Theta_{\lambda}^{-1}\left(
\xi\right)  \right)  ,
\]
we obtain%
\begin{align}
\mathrm{s}\left(  \Gamma_{H},\Gamma_{N},\mathbf{g,f}\right)   &  =\sum
_{\ell\in\Gamma_{H}}\int_{\ell^{-1}\mathcal{O}}\left\vert \mathbf{g}\left(
h\right)  \right\vert ^{2}\left\vert \mathbf{f}\left(  \ell h\right)
\mathbf{w}\left(  \ell h\right)  ^{1/2}\right\vert ^{2}d\mu_{H}\left(
h\right) \label{lastline}\\
&  =\sum_{\ell\in\Gamma_{H}}\int_{H}\left\vert \mathbf{g}\left(  h\right)
\right\vert ^{2}\left\vert \mathbf{f}\left(  \ell h\right)  \mathbf{w}\left(
\ell h\right)  ^{1/2}\right\vert ^{2}d\mu_{H}\left(  h\right)  .
\label{Lastline2}%
\end{align}
Using Tonelli's Theorem to interchange the sum and integral in
(\ref{Lastline2}) yields
\[
\mathrm{s}\left(  \Gamma_{H},\Gamma_{N},\mathbf{g,f}\right)  =\left\vert
C\right\vert \cdot\int_{H}\left\vert \mathbf{g}\left(  h\right)  \right\vert
^{2}\sum_{\ell\in\Gamma_{H}}\left\vert \mathbf{f}\left(  \ell h\right)
\mathbf{w}\left(  \ell h\right)  ^{1/2}\right\vert ^{2}d\mu_{H}\left(
h\right)  .
\]
In summary,
\[
\mathrm{s}\left(  \Gamma_{H},\Gamma_{N},\mathbf{g,f}\right)  \leq\left\vert
C\right\vert \cdot M_{H,\mathbf{f}}\int_{H}\left\vert \mathbf{g}\left(
h\right)  \right\vert ^{2}d\mu_{H}\left(  h\right)  =\left\vert C\right\vert
M_{H,\mathbf{f}}\left\Vert \mathbf{g}\right\Vert _{\mathfrak{H}}^{2}%
\]
and
\[
\mathrm{s}\left(  \Gamma_{H},\Gamma_{N},\mathbf{g,f}\right)  \geq\left\vert
C\right\vert \cdot m_{H,\mathbf{f}}\int_{H}\left\vert \mathbf{g}\left(
h\right)  \right\vert ^{2}d\mu_{H}\left(  h\right)  =\left\vert C\right\vert
m_{H,\mathbf{f}}\left\Vert \mathbf{g}\right\Vert _{\mathfrak{H}}^{2}.
\]
Since the set of all continuous and compactly supported functions is dense in
$\mathfrak{H}$, it follows that the collection
\[
\left\{  \pi\left(  \ell^{-1}\kappa\right)  \mathbf{f}:\left(  \ell
,\kappa\right)  \in\Gamma_{H}\times\Gamma_{N}\right\}
\]
is a frame for the Hilbert space $\mathfrak{H}$ with frame bounds $\left\vert
C\right\vert \cdot m_{H,\mathbf{f}}\leq\left\vert C\right\vert \cdot
M_{H,\mathbf{f}}.$ This proves the first part of the result.

For (1), assuming additionally that $m_{H,\mathbf{f}}=M_{H,\mathbf{f}},$ the
lower and upper frame bound coincide and consequently, $\mathcal{S}\left(
\mathbf{f,}\Gamma_{H}^{-1}\Gamma_{N}\right)  $ is a tight frame for
$\mathfrak{H}.$

Regarding (2), under the assumption that
\[
\left\vert C\right\vert \cdot m_{H,\mathbf{f}}=\left\vert C\right\vert \cdot
M_{H,\mathbf{f}}=1
\]
and $\int_{H}\left\vert \mathbf{f}\left(  h\right)  \right\vert ^{2}d\mu
_{H}\left(  h\right)  =1,$ we obtain $\mathcal{S}\left(  \mathbf{f,}\Gamma
_{H}^{-1}\Gamma_{N}\right)  $ is a unit norm Parseval frame and therefore, it
must be necessarily be an orthonormal basis.
\end{proof}

\begin{proposition}
Let $\mathbf{f}$ be a continuous (or a smooth) function which is compactly
supported on $\mathcal{O}$. Next, let $\Gamma_{N}$ be the discrete set
described in Condition \ref{Cond} (2). Given a discrete set $\Lambda$ of $H$,
if the system $\mathcal{S}\left(  \mathbf{f,}\Lambda^{-1}\Gamma_{N}\right)  $
is a frame for $\mathfrak{H}$ with frame bounds $A,B$ then for almost every
$h\in H,$ we have
\[
A\left\vert C\right\vert ^{-1}\leq\sum_{\ell\in\Lambda}\left\vert
\mathbf{f}\left(  \ell h\right)  \sqrt{\mathbf{w}\left(  \ell h\right)
}\right\vert ^{2}\leq B\left\vert C\right\vert ^{-1}.
\]
In particular, $\mathbf{f}$ must be bounded.
\end{proposition}

\begin{proof}
The proof goes as follows. Put
\[
\mathrm{s}\left(  \Lambda,\Gamma_{N},\mathbf{g,f}\right)  =\sum_{\ell
\in\Lambda}\sum_{\exp\left(  PX\right)  \in\Gamma_{N}}\left\vert \left\langle
\mathbf{g},\pi\left(  \ell^{-1}\right)  \pi\left(  \exp\left(  PX\right)
\right)  \mathbf{f}\right\rangle _{\mathfrak{H}}\right\vert ^{2}.
\]
By assumption, given $\mathbf{g}\in\mathfrak{H,}$
\begin{align*}
\mathrm{s}\left(  \Lambda,\Gamma_{N},\mathbf{g,f}\right)   &  =\int
_{H}\left\vert \mathbf{g}\left(  h\right)  \right\vert ^{2}\left(  \left\vert
C\right\vert \sum_{\ell\in\Lambda}\left\vert \mathbf{f}\left(  \ell h\right)
\mathbf{w}\left(  \ell h\right)  ^{1/2}\right\vert ^{2}\right)  d\mu
_{H}\left(  h\right) \\
&  \geq\int_{H}A\left\vert \mathbf{g}\left(  h\right)  \right\vert ^{2}%
d\mu_{H}\left(  h\right)  .
\end{align*}
As such,
\[
\int_{H}\left\vert \mathbf{g}\left(  h\right)  \right\vert ^{2}\left(
\left\vert C\right\vert \sum_{\ell\in\Lambda}\left\vert \mathbf{f}\left(  \ell
h\right)  \mathbf{w}\left(  \ell h\right)  ^{1/2}\right\vert ^{2}-A\right)
d\mu_{H}\left(  h\right)  \geq0.
\]
Suppose that
\[
\left\vert C\right\vert \sum_{\ell\in\Gamma_{H}}\left\vert \mathbf{f}\left(
\ell h\right)  \mathbf{w}\left(  \ell h\right)  ^{1/2}\right\vert ^{2}<A
\]
on some subset $E\subseteq H$ of positive and finite Haar measure. Letting
$\mathbf{g}$ be the indicator function of the set $E,$
\begin{align*}
\mathrm{s}\left(  \Lambda,\Gamma_{N},\mathbf{g,f}\right)   &  =\int
_{E}\left\vert C\right\vert \cdot\sum_{\ell\in\Lambda}\left\vert
\mathbf{f}\left(  \ell h\right)  \mathbf{w}\left(  \ell h\right)
^{1/2}\right\vert ^{2}d\mu_{H}\left(  h\right) \\
&  <A\cdot\int_{E}d\mu_{H}\left(  h\right) \\
&  =A\left\Vert \mathbf{g}\right\Vert ^{2}.
\end{align*}
Thus, $A\left\Vert \mathbf{g}\right\Vert ^{2}>\mathrm{s}\left(  \Lambda
,\Gamma_{N},\mathbf{g,f}\right)  .$ This contradicts our assumption.
Similarly, suppose that
\[
\left\vert C\right\vert \cdot\sum_{\ell\in\Lambda}\left\vert \mathbf{f}\left(
\ell h\right)  \mathbf{w}\left(  \ell h\right)  ^{1/2}\right\vert ^{2}>B
\]
on some subset $F$ of positive and finite Haar measure. Letting $\mathbf{g}$
be the indicator function of the set $F,$
\[
\sum_{\ell\in\Lambda}\sum_{\exp\left(  PX\right)  \in\Gamma_{N}}\left\vert
\left\langle \mathbf{g},\pi\left(  \ell^{-1}\right)  \pi\left(  \exp\left(
PX\right)  \right)  \mathbf{f}\right\rangle _{\mathfrak{H}}\right\vert
^{2}>B\left\Vert \mathbf{g}\right\Vert ^{2}%
\]
and this contradicts that $\mathcal{S}\left(  \mathbf{f,}\Lambda^{-1}%
\Gamma_{N}\right)  $ is a frame for $\mathfrak{H}$ with upper frame bound $B.$
\end{proof}

The following theorem offers concrete conditions for the existence of Parseval
frames and orthonormal bases of the type
\[
\mathcal{S}\left(  \mathbf{s}\cdot1_{A}\mathbf{,}\Gamma_{H}^{-1}\Gamma
_{N}\right)
\]
for some measurable function $\mathbf{s}$ defined on $\mathcal{O}$ and a
$d\mu_{H}$-measurable subset $A$ of $\mathcal{O}.$ \newpage

\begin{proposition}
\label{criteria} \text{ }

\begin{enumerate}
\item Let $\mathbf{f}$ be a vector in $L^{2}\left(  \mathcal{O},d\mu
_{H}\right)  $ satisfying Condition \ref{Cond} (7). If additionally, Condition
\ref{Cond} (2) and (6)\ hold then $\mathcal{S}\left(  \left\vert C\right\vert
^{-1/2}\mathbf{f,}\Gamma_{H}^{-1}\Gamma_{N}\right)  $ is a Parseval frame for
$\mathfrak{H}$.

\item Let $\mathbf{f}$ be a vector in $L^{2}\left(  \mathcal{O},d\mu
_{H}\right)  $ satisfying Condition \ref{Cond} (7) and (8). If additionally,
Condition \ref{Cond} (2) and (6)\ hold then $\mathcal{S}\left(  \left\vert
C\right\vert ^{-1/2}\mathbf{f,}\Gamma_{H}^{-1}\Gamma_{N}\right)  $ is an
orthonormal basis for $\mathfrak{H}.$
\end{enumerate}
\end{proposition}

\begin{proof}
First, we observe that if $\mathbf{f}$ is a function defined on $\mathcal{O}$
such that
\[
\mathbf{f}\left(  \varphi^{-1}\left(  \Theta_{\lambda}^{-1}\left(  \xi\right)
\right)  \right)  \mathbf{W}_{\lambda}\left(  \xi\right)  ^{1/2}%
=1_{\Theta_{\lambda}\left(  \varphi\left(  \mathcal{O}\right)  \right)
}\left(  \xi\right)  ,
\]
then $\mathbf{f}\in L^{2}\left(  \mathcal{O},d\mu_{H}\right)  .$ Indeed,
\begin{align*}
\int_{\mathcal{O}}\left\vert \mathbf{f}\left(  h\right)  \right\vert ^{2}%
d\mu_{H}\left(  h\right)   &  =\int_{\varphi\left(  \mathcal{O}\right)
}\left\vert \mathbf{f}\left(  \varphi^{-1}\left(  A\right)  \right)
\right\vert ^{2}d\mu_{H}\left(  \varphi^{-1}\left(  A\right)  \right) \\
^{\text{(see (\ref{W}))}}  &  =\int_{\Theta_{\lambda}\left(  \varphi\left(
\mathcal{O}\right)  \right)  }\left\vert \mathbf{f}\left(  \varphi^{-1}\left(
\Theta_{\lambda}^{-1}\left(  \xi\right)  \right)  \right)  \right\vert
^{2}\mathbf{W}_{\lambda}\left(  \xi\right)  d\xi
\end{align*}
and
\begin{align*}
\int_{\mathcal{O}}\left\vert \mathbf{f}\left(  h\right)  \right\vert ^{2}%
d\mu_{H}\left(  h\right)   &  =\int_{\Theta_{\lambda}\left(  \varphi\left(
\mathcal{O}\right)  \right)  }\left\vert \mathbf{W}_{\lambda}\left(
\xi\right)  ^{1/2}\mathbf{f}\left(  \varphi^{-1}\left(  \Theta_{\lambda}%
^{-1}\left(  \xi\right)  \right)  \right)  \right\vert ^{2}d\xi\\
&  =\left\vert \Theta_{\lambda}\left(  \varphi\left(  \mathcal{O}\right)
\right)  \right\vert <\infty.
\end{align*}
Given $X\in\mathfrak{p}$, and letting $\mathbf{g}$ be a continuous function
defined on $\mathcal{O}$, we obtain
\begin{align*}
I  &  =\left\langle \mathbf{g},\pi\left(  \exp X\right)  \mathbf{f}%
\right\rangle _{\mathfrak{H}}\\
&  =\int_{\mathcal{O}}\mathbf{g}\left(  h\right)  e^{2\pi i\left\langle
Ad\left(  h^{-1}\right)  ^{\ast}\lambda,X\right\rangle }\overline
{\mathbf{f}\left(  h\right)  }d\mu_{H}\left(  h\right) \\
^{\left(  \beta_{\lambda}\left(  \varphi^{-1}\left(  A\right)  \right)
=Ad\left(  h^{-1}\right)  ^{\ast}\lambda\right)  }  &  =\int_{\varphi\left(
\mathcal{O}\right)  }\mathbf{g}\left(  \varphi^{-1}\left(  A\right)  \right)
e^{2\pi i\left\langle \beta_{\lambda}\left(  \varphi^{-1}\left(  A\right)
\right)  ,X\right\rangle }\overline{\mathbf{f}\left(  \varphi^{-1}\left(
A\right)  \right)  }d\mu_{H}\left(  \varphi^{-1}\left(  A\right)  \right) \\
^{\text{Since }X=P^{\ast}X\in\mathfrak{p}}  &  =\int_{\varphi\left(
\mathcal{O}\right)  }\mathbf{g}\left(  \varphi^{-1}\left(  A\right)  \right)
e^{2\pi i\left\langle P^{\ast}\beta_{\lambda}\varphi^{-1}\left(  A\right)
,X\right\rangle }\overline{\mathbf{f}\left(  \varphi^{-1}\left(  A\right)
\right)  }d\mu_{H}\left(  \varphi^{-1}\left(  A\right)  \right) \\
^{\Theta_{\lambda}\left(  A\right)  =P^{\ast}\beta_{\lambda}\varphi
^{-1}\left(  A\right)  }  &  =\int_{\varphi\left(  \mathcal{O}\right)
}\mathbf{g}\left(  \varphi^{-1}\left(  A\right)  \right)  e^{2\pi
i\left\langle \Theta_{\lambda}\left(  A\right)  ,X\right\rangle }%
\overline{\mathbf{f}\left(  \varphi^{-1}\left(  A\right)  \right)  }d\mu
_{H}\left(  \varphi^{-1}\left(  A\right)  \right)  .
\end{align*}
Next, the change of variable
\[
\xi=\Theta_{\lambda}\left(  A\right)  \Leftrightarrow A=\Theta_{\lambda}%
^{-1}\left(  \xi\right)
\]
yields
\[
I=\int_{\Theta_{\lambda}\left(  \varphi\left(  \mathcal{O}\right)  \right)
}e^{2\pi i\left\langle \xi,X\right\rangle }\left(  \mathbf{g}\left(
\varphi^{-1}\left(  \Theta_{\lambda}^{-1}\left(  \xi\right)  \right)  \right)
\overline{\mathbf{f}\left(  \varphi^{-1}\left(  \Theta_{\lambda}^{-1}\left(
\xi\right)  \right)  \right)  }\mathbf{W}_{\lambda}\left(  \xi\right)
\right)  d\xi.
\]
Let $C$ be a compact subset of $\mathfrak{n}^{\ast}$ containing $\Theta
_{\lambda}\left(  \varphi\left(  \mathcal{O}\right)  \right)  \ $such that
\begin{equation}
\left\{  \frac{e^{2\pi i\left\langle \xi,X\right\rangle }\cdot1_{\Theta
_{\lambda}\left(  \varphi\left(  \mathcal{O}\right)  \right)  }\left(
\xi\right)  }{\left\vert C\right\vert ^{1/2}}:\exp\left(  X\right)  \in
\Gamma_{N}\right\}
\end{equation}
is a Parseval frame for $L^{2}\left(  \Theta_{\lambda}\left(  \varphi\left(
\mathcal{O}\right)  \right)  \right)  .$ Then%
\begin{align*}
I_{1}  &  =\sum_{\exp X\in\Gamma_{N}}\left\vert \left\langle \mathbf{g}%
,\pi\left(  \exp X\right)  \mathbf{f}\right\rangle _{\mathfrak{H}}\right\vert
^{2}\\
&  =\sum_{\exp X\in\Gamma_{N}}\left\vert \int_{\Theta_{\lambda}\left(
\varphi\left(  \mathcal{O}\right)  \right)  }\frac{e^{2\pi i\left\langle
\xi,X\right\rangle }}{\left\vert C\right\vert ^{1/2}}\mathbf{g}\left(
\varphi^{-1}\left(  \Theta_{\lambda}^{-1}\left(  \xi\right)  \right)  \right)
\left\vert C\right\vert ^{1/2}\mathbf{W}_{\lambda}\left(  \xi\right)
^{1/2}d\xi\right\vert ^{2}.
\end{align*}
Letting
\[
u\left(  \xi\right)  =\varphi^{-1}\left(  \Theta_{\lambda}^{-1}\left(
\xi\right)  \right)  ,
\]
we obtain%
\begin{align*}
I_{1}  &  =\int_{\Theta_{\lambda}\left(  \varphi\left(  \mathcal{O}\right)
\right)  }\left\vert \left\vert C\right\vert ^{1/2}\mathbf{g}\left(  u\left(
\xi\right)  \right)  \mathbf{W}_{\lambda}\left(  \xi\right)  ^{1/2}\right\vert
^{2}d\xi\\
&  =\int_{\Theta_{\lambda}\left(  \varphi\left(  \mathcal{O}\right)  \right)
}\left\vert C\right\vert \cdot\left\vert \mathbf{g}\left(  u\left(
\xi\right)  \right)  \right\vert ^{2}\mathbf{W}_{\lambda}\left(  \xi\right)
d\xi\\
&  =\left\vert C\right\vert \cdot\int_{\mathcal{O}}\left\vert \mathbf{g}%
\left(  h\right)  \right\vert ^{2}d\mu_{H}\left(  h\right)  .
\end{align*}
Assuming the existence of a discrete set $\Gamma_{H}\subset H$ such that
$\left\{  \ell^{-1}\mathcal{O}:\ell\in\Gamma_{H}\right\}  $ is a tiling of
$H$; since the operators $\pi\left(  \ell\right)  , \ell\in\Gamma_{H}$ are
unitary, it immediately follows that
\begin{align*}
\sum_{\ell\in\Gamma_{H}}\sum_{\exp X\in\Gamma_{N}}\left\vert \left\langle
\mathbf{g},\pi\left(  \ell^{-1}\right)  \pi\left(  \exp X\right)
\mathbf{f}\right\rangle \right\vert ^{2}  &  =\sum_{\ell\in\Gamma_{H}}%
\sum_{\exp X\in\Gamma_{N}}\left\vert \left\langle \pi\left(  \ell\right)
\mathbf{g}\cdot1_{\mathcal{O}},1_{\mathcal{O}}\cdot\pi\left(  \exp X\right)
\mathbf{f}\right\rangle \right\vert ^{2}\\
&  =\left\vert C\right\vert \cdot\sum_{\ell\in\Gamma_{H}}\left\Vert \pi\left(
\ell\right)  \left(  \mathbf{g}\cdot1_{\mathcal{O}}\right)  \right\Vert ^{2}\\
&  =\left\vert C\right\vert \cdot\sum_{\ell\in\Gamma_{H}}\int_{\mathcal{O}%
}\left\vert \mathbf{g}\left(  \ell^{-1}h\right)  \right\vert ^{2}d\mu
_{H}\left(  h\right) \\
&  =\left\vert C\right\vert \cdot\sum_{\ell\in\Gamma_{H}}\int_{\ell
^{-1}\mathcal{O}}\left\vert \mathbf{g}\left(  h\right)  \right\vert ^{2}%
d\mu_{H}\left(  h\right) \\
&  =\left\vert C\right\vert \cdot\int_{H}\left\vert \mathbf{g}\left(
h\right)  \right\vert ^{2}d\mu_{H}\left(  h\right)  .
\end{align*}
Therefore $\left\{  \pi\left(  \ell^{-1}\kappa\right)  \mathbf{f}:\kappa
\in\Gamma_{N},\ell\in\Gamma_{H}\right\}  $ is a tight frame for $L^{2}\left(
H,d\mu_{H}\left(  h\right)  \right)  $ with frame bound $\left\vert
C\right\vert .$ As a result,
\[
\left\{  \pi\left(  \ell^{-1}\kappa\right)  \left\vert C\right\vert
^{-1/2}\mathbf{f}:\kappa\in\Gamma_{N},\ell\in\Gamma_{H}\right\}
\]
is a Parseval frame.

If additionally,
\begin{equation}
\left\vert C\right\vert ^{-1/2}\cdot\left\Vert \mathbf{f}\right\Vert
_{\mathfrak{H}}=\left\vert C\right\vert ^{-1/2}\cdot\left\vert \Theta
_{\lambda}\left(  \varphi\left(  \mathcal{O}\right)  \right)  \right\vert
^{1/2}=1 \label{stepA}%
\end{equation}
then
\begin{align*}
\left\Vert \left\vert C\right\vert ^{-1/2}\mathbf{f}\right\Vert _{\mathfrak{H}%
}  &  =\left(  \left\vert C\right\vert ^{-1}\int_{\mathcal{O}}\left\vert
\mathbf{f}\left(  x\right)  \right\vert ^{2}d\mu_{H}\left(  x\right)  \right)
^{1/2}\\
&  =\left(  \left\vert C\right\vert ^{-1}\int_{\Theta_{\lambda}\left(
\varphi\left(  \mathcal{O}\right)  \right)  }\left\vert \mathbf{f}\left(
\varphi^{-1}\left(  \Theta_{\lambda}^{-1}\left(  \xi\right)  \right)  \right)
\right\vert ^{2}d\mu_{H}\left(  \varphi^{-1}\left(  \Theta_{\lambda}%
^{-1}\left(  \xi\right)  \right)  \right)  \right)  ^{1/2}\\
&  =\left(  \left\vert C\right\vert ^{-1}\int_{\Theta_{\lambda}\left(
\varphi\left(  \mathcal{O}\right)  \right)  }\left\vert \mathbf{f}\left(
\varphi^{-1}\left(  \Theta_{\lambda}^{-1}\left(  \xi\right)  \right)  \right)
\mathbf{W}_{\lambda}\left(  \xi\right)  ^{1/2}\right\vert ^{2}d\xi\right)
^{1/2}\\
&  =\left(  \left\vert C\right\vert ^{-1}\int_{\Theta_{\lambda}\left(
\varphi\left(  \mathcal{O}\right)  \right)  }\frac{1_{\Theta_{\lambda}\left(
\varphi\left(  \mathcal{O}\right)  \right)  }\left(  \xi\right)  }{\left(
\mathbf{W}_{\lambda}\left(  \xi\right)  \right)  ^{1/2}}\left(  \mathbf{W}%
_{\lambda}\left(  \xi\right)  \right)  ^{1/2}d\xi\right)  ^{1/2}\\
&  =\left(  \left\vert C\right\vert ^{-1}\int_{\Theta_{\lambda}\left(
\varphi\left(  \mathcal{O}\right)  \right)  }d\xi\right)  ^{1/2}\\
^{\text{By (\ref{stepA})}}  &  =\left\vert C\right\vert ^{-1/2}\cdot\left\vert
\Theta_{\lambda}\left(  \varphi\left(  \mathcal{O}\right)  \right)
\right\vert ^{1/2}=1.
\end{align*}
Therefore,
\[
\left\{  \pi\left(  \ell^{-1}\kappa\right)  \left(  \left\vert C\right\vert
^{-1/2}\mathbf{f}\right)  :\kappa\in\Gamma_{N},\ell\in\Gamma_{H}\right\}
\]
is an orthonormal basis for $\mathfrak{H}.$
\end{proof}

We are now in position to establish our main results.

\subsection{Proof of Theorem \ref{main}}

Appealing to Lemma \ref{communicated} and Lemma \ref{tstar}, we see that the
sufficient conditions stated of Lemma \ref{FProperty} hold for any function
$\mathbf{f}$ which is continuous (or smooth) and supported on $\mathcal{O}.$

\subsection{Proof of Theorem \ref{PFEXP}}

We will start this section by reviewing some fundamental concepts of
exponential solvable Lie groups \cite{varadarajan2013lie}. Let $\mathfrak{h}$
be a real solvable Lie algebra of dimension $r.$ We define the derived series
of $\mathfrak{h}$ inductively as follows
\[
D^{1}\mathfrak{h}=\left[  \mathfrak{h,h}\right]  ,\ D^{k}\mathfrak{h}=\left[
D^{k-1}\mathfrak{h,}D^{k-1}\mathfrak{h}\right]  .
\]
Here $\left[  \mathfrak{h,h}\right]  $ is the real span of $\left[
X,Y\right]  $ for $X,Y\in\mathfrak{h}$. Moreover, we say that $\mathfrak{h}$
is solvable if there exists a natural number $m$ such that $\dim\left(
D^{m}\mathfrak{h}\right)  =0.$ The smallest such natural number $m$ is called
the length of $\mathfrak{h}.$ For example, it is not hard to verify that the
algebra of upper triangular matrices is a solvable Lie algebra.

Next, for a Lie algebra $\mathfrak{h,}$ the descending central series is
defined by the sequence
\[
\mathfrak{h}^{\left(  0\right)  }=\mathfrak{h,h}^{\left(  1\right)  }=\left[
\mathfrak{h,h}\right]  ,\cdots,\mathfrak{h}^{\left(  k\right)  }=\left[
\mathfrak{h}^{\left(  k-1\right)  }\mathfrak{,h}\right]
\]
and $\mathfrak{h}$ is called nilpotent if there exists a natural number $m$
such that $\dim\mathfrak{h}^{\left(  m\right)  }=0.$ It is easy to verify that
all nilpotent Lie algebras are solvable. However, the containment of nilpotent
Lie algebras in the set of solvable algebras is strictly proper.

A connected Lie group $H$ is said to be solvable if its Lie algebra is
solvable. Keep in mind that even if $H$ is simply connected and solvable, the
exponential map may still fail to be bijective. For a solvable Lie group $H$
with Lie algebra $\mathfrak{h,}$ when the corresponding exponential map is a
bijection, we say that such a group is \textbf{exponential} and its Lie
algebra is also called exponential. Such a group is necessarily connected and
simply connected as well.

The following lemmas will be instrumental

\begin{lemma}
\label{VAR0}\cite[Corollary 3.7.5.]{varadarajan2013lie} Let $\mathfrak{h}$ be
a solvable Lie algebra over the reals. Then we can find subalgebras
$\mathfrak{h=h}_{1},\mathfrak{h}_{2},\cdots,\mathfrak{h}_{r}=\left\{
0\right\}  $ such that (1) Each $\mathfrak{h}_{k+1}\subseteq\mathfrak{h}_{k}$
and $\mathfrak{h}_{k+1}$ is an ideal of $\mathfrak{h}_{k}$ and (2)
$\dim\left(  \mathfrak{h}_{k}/\mathfrak{h}_{k+1}\right)  =1$ for $k\in\left\{
1,\cdots,r\right\}  .$
\end{lemma}

\begin{lemma}
\label{VAR}\cite[Theorem 3.18.11]{varadarajan2013lie} Let $H$ be a simply
connected solvable Lie group with Lie algebra $\mathfrak{h.}$ Suppose that
$\left\{  A_{1},\cdots,A_{r}\right\}  $ is a basis \footnote{The existence of
such a basis a direct consequence of Lemma \ref{VAR0}} for $\mathfrak{h}$
satisfying the following property: $\mathfrak{h}_{\left(  j\right)  }%
=\sum_{k=1}^{j}%
\mathbb{R}
A_{k}$ is a subalgebra of $\mathfrak{h}$ and each $\mathfrak{h}_{\left(
i\right)  }$ is an ideal in $\mathfrak{h}_{\left(  i+1\right)  }.$ Then the
map%
\[
\left(  a_{1},\cdots,a_{r}\right)  \mapsto\exp\left(  a_{1}A_{1}\right)
\cdots\exp\left(  a_{r}A_{r}\right)
\]
is an analytic diffeomorphism of $%
\mathbb{R}
^{r}$ onto $H.$
\end{lemma}

Moving forward, we will also need the following lemma. Let $\mathcal{E}\left(
C,\Gamma_{N}\right)  $ be as defined in (\ref{EC}).

\begin{lemma}
\label{lemmata}Let $\mathbf{s}$ be a continuous (or a smooth) function which
is supported on $\mathcal{O}.$ Suppose that $\Gamma_{H}$ is a discrete subset
of $H$ such that
\[
\sum_{\ell\in\Gamma_{H}}\left\vert \mathbf{s}\left(  \ell h\right)
\right\vert ^{2}=1
\]
for almost every $h\in H.$ Let $C$ be a compact subset of $\mathfrak{p}^{\ast
}$ such that $\Theta_{\lambda}\left(  \varphi\left(  \mathcal{O}\right)
\right)  \subseteq C$ and $\Gamma_{N}$ is a discrete subset of $\exp\left(
\mathfrak{p}\right)  $ such that the trigonometric system $\mathcal{E}\left(
C,\Gamma_{N}\right)  $ is an orthonormal basis for $L^{2}\left(
C,d\xi\right)  .$ Then the system
\[
\left\{  \left[  \pi\left(  \ell^{-1}\kappa\right)  \right]  \left\vert
C\right\vert ^{-1/2}\mathbf{sw}^{-1/2}:\left(  \ell,\kappa\right)  \in
\Gamma_{H}\times\Gamma_{N}\right\}
\]
is a Parseval frame generated by the compactly supported and continuous
function $\mathbf{sw}^{-1/2}\left\vert C\right\vert ^{-1/2}.$
\end{lemma}

\begin{proof}
Let us suppose that $\mathbf{s}$ is a continuous (or a smooth) function which
is compactly supported on $\mathcal{O}$ such that $\sum_{\ell\in\Gamma_{H}%
}\left\vert \mathbf{s}\left(  \ell h\right)  \right\vert ^{2}=1$ for all $h\in
H.$ Next, let $\mathbf{f}=\mathbf{sw}^{-1/2}.$ Then
\[
\sum_{\kappa\in\Gamma_{H}}\left\vert \mathbf{f}\left(  \kappa h\right)
\sqrt{\mathbf{w}\left(  \kappa h\right)  }\right\vert ^{2}=\sum_{\kappa
\in\Gamma_{H}}\left\vert \frac{\mathbf{s}\left(  \kappa h\right)  }%
{\sqrt{\mathbf{w}\left(  \kappa h\right)  }}\sqrt{\mathbf{w}\left(  \kappa
h\right)  }\right\vert ^{2}=\sum_{\kappa\in\Gamma_{H}}\left\vert
\mathbf{s}\left(  kh\right)  \right\vert ^{2}=1.
\]
Appealing to Lemma \ref{FProperty} Part (1), the stated result is immediate.
\end{proof}

Suppose that $H$ is a connected exponential solvable Lie group. We will prove
by induction on the dimension of $H$ that there exists a continuous function
$\mathbf{s}$ compactly supported in $\mathcal{O}$ such that
\[
\sum_{\ell\in\Gamma_{H}}\left\vert \mathbf{s}\left(  \ell h\right)
\right\vert ^{2}=1
\]
for every element $h$ in $H.$

For the base case, let us suppose that the dimension of $H$ is equal to $1.$
Moreover, let us suppose that $\mathcal{O}=\left(  -\epsilon^{-1}%
,\epsilon^{-1}\right)  $ for some positive real number $\epsilon.$ Put
\[
\mathbf{s}\left(  t\right)  =\left(  1_{\left[  -1/2,1/2\right]  }%
\ast1_{\left[  -1/2,1/2\right]  }\right)  ^{\frac{1}{2}}\left(  \epsilon
t\right)  .
\]
In the above, $\ast$ stands for the usual convolution product for functions
defined on the real line given by
\[
\left(  f\ast g\right)  \left(  y\right)  =\int_{%
\mathbb{R}
}f\left(  x\right)  g\left(  y-x\right)  dx.
\]
Since
\[
\sum_{k\in%
\mathbb{Z}
}\left(  1_{\left[  -1/2,1/2\right]  }\ast1_{\left[  -1/2,1/2\right]
}\right)  \left(  t+k\right)  =1\text{ for all }t\in%
\mathbb{R}
,
\]
it follows that
\[
\sum_{k\in%
\mathbb{Z}
}\left\vert \mathbf{s}\left(  t+\epsilon^{-1}k\right)  \right\vert ^{2}%
=\sum_{k\in%
\mathbb{Z}
}\left(  1_{\left[  -1/2,1/2\right]  }\ast1_{\left[  -1/2,1/2\right]
}\right)  \left(  t\epsilon+k\right)  =1
\]
holds for every real number $t.$ Note that the smoothness of $\mathbf{s}$ can
be improved if $\mathbf{s}$ is replaced by a suitable spline-type function of
higher order. To this end, let $\mathbf{b}_{n}$ be the function obtained by
convolving the function $1_{\left[  -1/2,1/2\right]  }$ with itself $n$-many
times. Next, let
\[
\mathbf{s}_{n}\left(  t\right)  =\left(  \mathbf{b}_{n}\right)  ^{\frac{1}{2}%
}\left(  \frac{\epsilon n}{2}t\right)  .
\]
Then
\begin{align*}
\sum_{k\in%
\mathbb{Z}
}\left\vert \mathbf{s}_{n}\left(  t+\frac{2}{\epsilon n}k\right)  \right\vert
^{2}  &  =\sum_{k\in%
\mathbb{Z}
}\left\vert \left(  \mathbf{b}_{n}\right)  ^{\frac{1}{2}}\left(
\frac{\epsilon n}{2}\left(  t+\frac{2}{\epsilon n}k\right)  \right)
\right\vert ^{2}\\
&  =\sum_{k\in%
\mathbb{Z}
}\mathbf{b}_{n}\left(  k+\frac{nt\epsilon}{2}\right) \\
&  =1
\end{align*}
for every $t\in%
\mathbb{R}
$ as desired.

Next, let us assume that our result holds true whenever, $\dim H\leq r$ for
some natural number $r\geq1.$ Suppose next that $\dim\left(  H\right)  =r+1.$
Since $H$ is an exponential solvable Lie group, we may write $H=H_{\mathbf{n}%
}H_{\mathbf{c}}$ such that $H_{\mathbf{n}}$ is normal in $H$ and
$H_{\mathbf{c}}$ is a 1-parameter subgroup of $H$ (by appealing to Lemma
\ref{VAR0} and Lemma \ref{VAR}) which we shall identify with the set of real
numbers. Without loss of generality, we may assume that $\mathcal{O=O}%
_{\mathbf{n}}\mathcal{O}_{\mathbf{c}}$ such that $\mathcal{O}_{\mathbf{n}%
}\subset$ $H_{\mathbf{n}}$ and $\mathcal{O}_{\mathbf{c}}\subset H_{\mathbf{c}%
}\ $such that (inductive hypothesis) the following holds true. There exist
$\Gamma_{\mathbf{n}}\subset H_{\mathbf{n}}$ and $\Gamma_{\mathbf{c}}\subset
H_{\mathbf{c}}$ and continuous functions (or smooth functions) $\mathbf{s}%
_{\mathbf{n}}$ and $\mathbf{s}_{\mathbf{c}}$ such that
\[
\sum_{\ell_{\mathbf{n}}\in\Gamma_{\mathbf{n}}}\left\vert \mathbf{s}%
_{\mathbf{n}}\left(  \ell_{\mathbf{n}}h^{\prime}\right)  \right\vert ^{2}=1
\]
$\text{for almost every }h^{\prime}\in H_{\mathbf{n}}$ and
\[
\sum_{\ell_{\mathbf{c}}\in\Gamma_{\mathbf{c}}}\left\vert \mathbf{s}%
_{\mathbf{c}}\left(  \ell_{\mathbf{c}}h^{\prime\prime}\right)  \right\vert
^{2}=1
\]
$\text{ for almost every }h^{\prime\prime}\in H_{\mathbf{c}}.$ Next, we define
a continuous function $\mathbf{s}$ on $H$ as follows. Given $h\in H,$ since we
can uniquely factor $h$ such that $h=h_{\mathbf{n}}h_{\mathbf{c}}$ with
$h_{\mathbf{n}}\in H_{\mathbf{n}}$ and $h_{\mathbf{c}}\in H_{\mathbf{c}},$ we
define
\[
\mathbf{s}\left(  h\right)  =\mathbf{s}\left(  h_{\mathbf{n}}h_{\mathbf{c}%
}\right)  =\mathbf{s}_{\mathbf{n}}\left(  h_{\mathbf{n}}\right)
\cdot\mathbf{s}_{\mathbf{c}}\left(  h_{\mathbf{c}}\right)
\]
for all $h$ in $H.$ Thus, for $h\in H,$
\[
\sum_{\ell_{\mathbf{c}}\in\Gamma_{\mathbf{c}}}\sum_{\ell_{\mathbf{n}}\in
\Gamma_{\mathbf{n}}}\left\vert \mathbf{s}\left(  \ell_{\mathbf{n}}%
\ell_{\mathbf{c}}h\right)  \right\vert ^{2}=\sum_{\ell_{\mathbf{c}}\in
\Gamma_{\mathbf{c}}}\sum_{\ell_{\mathbf{n}}\in\Gamma_{\mathbf{n}}}\left\vert
\mathbf{s}\left(  \ell_{\mathbf{n}}\ell_{\mathbf{c}}h_{\mathbf{n}%
}h_{\mathbf{c}}\right)  \right\vert ^{2}.
\]
Rewriting the identity in $H$ as $e=\ell_{\mathbf{c}}^{-1}\ell_{\mathbf{c}},$
we obtain:
\[
\sum_{\ell_{\mathbf{c}}\in\Gamma_{\mathbf{c}}}\sum_{\ell_{\mathbf{n}}\in
\Gamma_{\mathbf{n}}}\left\vert \mathbf{s}\left(  \ell_{\mathbf{n}}%
\ell_{\mathbf{c}}h\right)  \right\vert ^{2}=\sum_{\ell_{\mathbf{c}}\in
\Gamma_{\mathbf{c}}}\sum_{\ell_{\mathbf{n}}\in\Gamma_{\mathbf{n}}}\left\vert
\mathbf{s}\left(  \ell_{\mathbf{n}}\ell_{\mathbf{c}}h_{\mathbf{n}}\left(
\ell_{\mathbf{c}}^{-1}\ell_{\mathbf{c}}\right)  h_{\mathbf{c}}\right)
\right\vert ^{2}.
\]
Next, appealing to the normality of $H_{\mathbf{n}}$ in $H,$ it is clear that
$\ell_{\mathbf{n}}\ell_{\mathbf{c}}h_{\mathbf{n}}\ell_{\mathbf{c}}^{-1}\in
H_{\mathbf{n}}.$ Moving forward,
\[
\sum_{\ell_{\mathbf{c}}\in\Gamma_{\mathbf{c}}}\sum_{\ell_{\mathbf{n}}\in
\Gamma_{\mathbf{n}}}\left\vert \mathbf{s}\left(  \left(  \ell_{\mathbf{n}}%
\ell_{\mathbf{c}}h_{\mathbf{n}}\ell_{\mathbf{c}}^{-1}\right)  \left(
\ell_{\mathbf{c}}h_{\mathbf{c}}\right)  \right)  \right\vert ^{2}=\sum
_{\ell_{\mathbf{c}}\in\Gamma_{\mathbf{c}}}\sum_{\ell_{\mathbf{n}}\in
\Gamma_{\mathbf{n}}}\left\vert \mathbf{s}_{\mathbf{n}}\left(  \ell
_{\mathbf{n}}\ell_{\mathbf{c}}h_{\mathbf{n}}\ell_{\mathbf{c}}^{-1}\right)
\right\vert ^{2}\cdot\left\vert \mathbf{s}_{\mathbf{c}}\left(  \ell
_{\mathbf{c}}h_{\mathbf{c}}\right)  \right\vert ^{2}.
\]
However, by assumption, for any fixed $\ell_{\mathbf{c}}\in\Gamma_{\mathbf{c}%
},$ since $\ell_{\mathbf{c}}h_{\mathbf{n}}\ell_{\mathbf{c}}^{-1}\in
H_{\mathbf{n}},$ it is clearly the case that $\sum_{\ell_{\mathbf{n}}\in
\Gamma_{\mathbf{n}}}\left\vert \mathbf{s}_{\mathbf{n}}\left(  \ell
_{\mathbf{n}}\ell_{\mathbf{c}}h_{\mathbf{n}}\ell_{\mathbf{c}}^{-1}\right)
\right\vert ^{2}=1.$ Therefore,
\[
\sum_{\ell_{\mathbf{c}}\in\Gamma_{\mathbf{c}}}\sum_{\ell_{\mathbf{n}}\in
\Gamma_{\mathbf{n}}}\left\vert \mathbf{s}_{\mathbf{n}}\left(  \ell
_{\mathbf{n}}\ell_{\mathbf{c}}h_{\mathbf{n}}\ell_{\mathbf{c}}^{-1}\right)
\right\vert ^{2}\cdot\left\vert \mathbf{s}_{\mathbf{c}}\left(  \ell
_{\mathbf{c}}h_{\mathbf{c}}\right)  \right\vert ^{2}=\sum_{\ell_{\mathbf{c}%
}\in\Gamma_{\mathbf{c}}}\left\vert \mathbf{s}_{\mathbf{c}}\left(
\ell_{\mathbf{c}}h_{\mathbf{c}}\right)  \right\vert ^{2}=1,
\]
and we conclude that
\[
\sum_{\ell_{\mathbf{c}}\in\Gamma_{\mathbf{c}}}\sum_{\ell_{\mathbf{n}}\in
\Gamma_{\mathbf{n}}}\left\vert \mathbf{s}\left(  \ell_{\mathbf{n}}%
\ell_{\mathbf{c}}h\right)  \right\vert ^{2}=1
\]
for all $h$ in $H.$

In light of Lemma \ref{lemmata}, Theorem \ref{PFEXP} is immediate.

\section{Examples and applications}

\subsection{A class of matrix groups satisfying the assumptions}

Proposition \ref{embed} gives us a procedure for constructing an extensive
collection of matrix groups satisfying the conditions listed in Section
\ref{assump}. We thank Gestur Olafsson for showing us the class of linear
groups described in Proposition \ref{embed}. This class of groups is inspired
by some earlier work of J. Wolf on the representation theory of maximal
parabolic subgroups of classical Lie groups \cite{wolf1976unitary}.

\begin{proposition}
\label{embed} Let%
\[
H=\left\{  \left[
\begin{array}
[c]{cc}%
h & \mathrm{0}_{n}\\
\mathrm{0}_{n} & \mathrm{Id}_{n}%
\end{array}
\right]  :h\in K\right\}  \text{ and }N=\left\{  \left[
\begin{array}
[c]{cc}%
\mathrm{Id}_{n} & x\\
\mathrm{0}_{n} & \mathrm{Id}_{n}%
\end{array}
\right]  :x\in\mathfrak{gl}\left(  n,%
\mathbb{R}
\right)  \right\}
\]
where $K$ is a closed and connected matrix subgroup of $GL\left(
n,\mathbb{R}\right)  .$ Put $G=NH.$ Then there exists a linear functional
$\lambda\in\mathfrak{n}^{\ast}$ such that the smooth map $h\mapsto
\beta_{\lambda}\left(  h\right)  =Ad\left(  h^{-1}\right)  ^{\ast}\lambda$ is
an immersion at the identity of $H$.
\end{proposition}

\begin{proof}
Let
\[
G=\left\{  \left[
\begin{array}
[c]{cc}%
h & x\\
\mathrm{0}_{n} & \mathrm{Id}_{n}%
\end{array}
\right]  :h\in K\text{ and }x\in\mathfrak{gl}\left(  n,%
\mathbb{R}
\right)  \right\}  .
\]
Clearly, the map
\[
h\mapsto\left[
\begin{array}
[c]{cc}%
h & \mathrm{0}_{n}\\
\mathrm{0}_{n} & \mathrm{Id}_{n}%
\end{array}
\right]
\]
defines a Lie group isomorphism between $K$ and $H$ and we shall identify $K$
with $H$ via this isomorphism. By assumption, the dimension of $N$ must be
greater than or equal to that of $H,$ and the Lie algebra of $N$ takes the
form
\[
\mathfrak{n}=\left\{  \left[
\begin{array}
[c]{cc}%
\mathrm{0}_{n} & x\\
\mathrm{0}_{n} & \mathrm{0}_{n}%
\end{array}
\right]  :x\in\mathfrak{gl}\left(  n,%
\mathbb{R}
\right)  \right\}  .
\]
We identify $\mathfrak{n}$ with its dual such that for $\lambda\in
\mathfrak{n}^{\ast}$ and $X\in\mathfrak{n,}$ the pairing $\left\langle
\lambda,X\right\rangle $ is given by the trace of the matrix $\lambda^{T}X.$
For $\lambda\in\mathfrak{n}^{\ast},$ the map
\[
\beta_{\lambda}:H\rightarrow\beta_{\lambda}\left(  H\right)  \subseteq
\mathfrak{n}^{\ast}%
\]
described in (\ref{Beta}) is computed as follows. Put
\[
\lambda_{\omega}=\lambda=\left[
\begin{array}
[c]{cc}%
\mathrm{0}_{n} & \omega\\
\mathrm{0}_{n} & \mathrm{0}_{n}%
\end{array}
\right]  \in\mathfrak{n}^{\ast}\text{ }%
\]
where $\omega$ is a matrix of order $n$ and let $h\in H.$ Then
\[
\left\langle \lambda,Ad\left(  h^{-1}\right)  X\right\rangle =\left\langle
\left[
\begin{array}
[c]{cc}%
\mathrm{0}_{n} & \omega\\
\mathrm{0}_{n} & \mathrm{0}_{n}%
\end{array}
\right]  ,\left[
\begin{array}
[c]{cc}%
h^{-1} & \mathrm{0}_{n}\\
\mathrm{0}_{n} & \mathrm{Id}_{n}%
\end{array}
\right]  \left[
\begin{array}
[c]{cc}%
\mathrm{0}_{n} & x\\
\mathrm{0}_{n} & \mathrm{0}_{n}%
\end{array}
\right]  \right\rangle .
\]
Suppose next, that $\omega$ is the identity matrix of order $n.$ Then
\[
\left\langle \lambda,Ad\left(  h^{-1}\right)  X\right\rangle =\mathrm{Tr}%
\left(  \underset{\lambda^{T}}{\underbrace{\left[
\begin{array}
[c]{cc}%
\mathrm{0}_{n} & \mathrm{0}_{n}\\
\mathrm{Id}_{n} & \mathrm{0}_{n}%
\end{array}
\right]  }}\underset{Ad\left(  h^{-1}\right)  X}{\underbrace{\left[
\begin{array}
[c]{cc}%
h^{-1} & \mathrm{0}_{n}\\
\mathrm{0}_{n} & \mathrm{Id}_{n}%
\end{array}
\right]  \left[
\begin{array}
[c]{cc}%
\mathrm{0}_{n} & x\\
\mathrm{0}_{n} & \mathrm{0}_{n}%
\end{array}
\right]  }}\right)  .
\]
Since
\[
\left[
\begin{array}
[c]{cc}%
\mathrm{0}_{n} & \mathrm{0}_{n}\\
\mathrm{Id}_{n} & \mathrm{0}_{n}%
\end{array}
\right]  \left[
\begin{array}
[c]{cc}%
h^{-1} & \mathrm{0}_{n}\\
\mathrm{0}_{n} & \mathrm{Id}_{n}%
\end{array}
\right]  =\left[
\begin{array}
[c]{cc}%
\mathrm{0}_{n} & \mathrm{0}_{n}\\
h^{-1} & \mathrm{0}_{n}%
\end{array}
\right]  ,
\]
it follows that
\[
\left\langle \lambda,Ad\left(  h^{-1}\right)  X\right\rangle =\mathrm{Tr}%
\left(  \left[
\begin{array}
[c]{cc}%
\mathrm{0}_{n} & \mathrm{0}_{n}\\
h^{-1} & \mathrm{0}_{n}%
\end{array}
\right]  \left[
\begin{array}
[c]{cc}%
\mathrm{0}_{n} & x\\
\mathrm{0}_{n} & \mathrm{0}_{n}%
\end{array}
\right]  \right)  .
\]
As such,
\[
\left\langle Ad\left(  h^{-1}\right)  ^{\ast}\lambda,X\right\rangle
=\left\langle \left[
\begin{array}
[c]{cc}%
\mathrm{0}_{n} & \left(  h^{-1}\right)  ^{T}\\
\mathrm{0}_{n} & \mathrm{0}_{n}%
\end{array}
\right]  ,\left[
\begin{array}
[c]{cc}%
\mathrm{0}_{n} & x\\
\mathrm{0}_{n} & \mathrm{0}_{n}%
\end{array}
\right]  \right\rangle
\]
and we obtain%
\[
Ad\left(  h^{-1}\right)  ^{\ast}\lambda=\left[
\begin{array}
[c]{cc}%
\mathrm{0}_{n} & \left(  h^{-1}\right)  ^{T}\\
\mathrm{0}_{n} & \mathrm{0}_{n}%
\end{array}
\right]  .
\]
By a slight abuse of notation, we can now say that $\beta_{\lambda}\left(
h\right)  =\left(  h^{-1}\right)  ^{T}$ and clearly, this implies that
$\beta_{\lambda}$ is an immersion at the identity of $H.$
\end{proof}

\begin{remark}
Note that the proof of Proposition \ref{embed} still works if we were to
replace $\lambda_{\omega}\in\mathfrak{n}^{\ast}$ with a linear function
$\lambda_{\omega^{\prime}}$ such that $\omega^{\prime}$ is any other
invertible matrix.
\end{remark}

\subsection{Some explicit formulas for $\mathbf{W}_{\lambda}$ and
applications}

\label{W}

Since $\mathbf{W}_{\lambda}$ plays a central role in this work, we exhibit
below some examples in which $\mathbf{W}_{\lambda}$ is explicitly computed. In
some cases, we also give a construction of frames and orthonormal bases.

\begin{example}
\label{solv} Let $G=%
\mathbb{R}
^{3}\rtimes%
\mathbb{R}
$ be a semidirect product equipped with the product
\[
\left(  x,t\right)  \left(  y,s\right)  =\left(  x+\left[
\begin{array}
[c]{ccc}%
\cos t & -\sin t & 0\\
\sin t & \cos t & 0\\
0 & 0 & e^{t}%
\end{array}
\right]  y,t+s\right)  .
\]
The group $G$ is simply connected and solvable but it is not completely
solvable. Next, let $\lambda\in\mathfrak{n}^{\ast}$ such that $\lambda=\left(
0,1,0\right)  .$ The corresponding induced representation $\pi$ is neither
irreducible nor square-integrable neither. The non irreducibility of $\pi$
follows from the fact that the stabilizer subgroup $H_{\lambda}$ in $H$ is not
trivial \cite[Chapter 6]{folland2016course}. Next, $Ad\left(  h^{-1}\right)
^{\ast}\lambda=\left(  -\sin t,\cos t,0\right)  \ $and letting $P^{\ast
}:\mathfrak{n}^{\ast}\rightarrow\mathfrak{n}^{\ast}$ such that $P^{\ast
}\left(  \lambda_{1},\lambda_{2},\lambda_{3}\right)  =\left(  \lambda
_{1},0,0\right)  ,$ we obtain $\Theta_{\lambda}\left(  t\right)  =-\sin t$
which is clearly a local diffeomorphism at zero (clearly not a global
diffeomorphism.) Its local inverse is given by $\Theta_{\lambda}^{-1}\left(
\xi\right)  =-\arcsin\left(  \xi\right)  $. As a result,
\[
\mathbf{W}_{\lambda}\left(  \xi\right)  =\frac{1}{\sqrt{1-\xi^{2}}}.
\]
Staying away from the singular points of $\Theta_{\lambda},$ fix
$\mathcal{O}=\left(  -\frac{\pi}{4},\frac{\pi}{4}\right)  .$ Then
$\Theta_{\lambda}$ defines a diffeomorphism between $\mathcal{O}$ and its
image. Let $\mathbf{f}$ be a continuous function supported on $\left(
-\frac{\pi}{4},\frac{\pi}{4}\right)  $ such that
\[
m_{H,\mathbf{f}}=\inf_{t}\left(  \sum_{\ell\in\alpha%
\mathbb{Z}
}\left\vert \frac{\mathbf{f}\left(  t+\ell\right)  }{\sqrt{1-\sin\left(
t+\ell\right)  ^{2}}}\right\vert ^{2}\right)  >0
\]
and
\[
M_{H,\mathbf{f}}=\sup_{t}\left(  \sum_{\ell\in\alpha%
\mathbb{Z}
}\left\vert \frac{\mathbf{f}\left(  t+\ell\right)  }{\sqrt{1-\sin\left(
t+\ell\right)  ^{2}}}\right\vert ^{2}\right)  <\infty
\]
for some positive real number $\alpha.$ Take for example,
\[
\mathbf{f}\left(  t\right)  =\sqrt{1-\sin\left(  t\right)  ^{2}}\left(
1_{\left(  -\frac{1}{2},\frac{1}{2}\right)  }\ast1_{\left(  -\frac{1}{2}%
,\frac{1}{2}\right)  }\right)  ^{1/2}\left(  t\right)  \text{ and }\alpha=1.
\]
Note also that $\left\{  \frac{e^{2\pi i\left\langle \xi,Y\right\rangle }%
}{2^{1/4}}:Y\in\frac{1}{\sqrt{2}}%
\mathbb{Z}
\right\}  $ is an orthonormal basis for $L^{2}\left(  C,d\xi\right)  $ where
$C=\sin\left(  -\frac{\pi}{4},\frac{\pi}{4}\right)  =\left(  -\frac{\sqrt{2}%
}{2},\frac{\sqrt{2}}{2}\right)  .$ Let
\[
\Gamma=\left\{  \left(  0,0,\ell\right)  \left(  \kappa,0,0\right)  :\ell
\in\alpha%
\mathbb{Z}
,\kappa\in\frac{1}{\sqrt{2}}%
\mathbb{Z}
\right\}  \subset G.
\]
By Proposition \ref{FProperty}, the system $\mathcal{S}\left(  \mathbf{f,}%
\Gamma\right)  $ is a frame for $L^{2}\left(
\mathbb{R}
\right)  $ with frame bounds $0<m_{H,\mathbf{f}}\cdot\sqrt{2}\leq
M_{H,\mathbf{f}}\cdot\sqrt{2}<\infty.$
\end{example}

\begin{example}
\label{shearlets}(A Toeplitz shearlet group, \cite{dahlke2012coorbit}) Let
$G=\mathbb{R}^{2}\rtimes\mathbb{R}^{2}$ be a semi-direct product group with
multiplication law given by $\left(  v,t\right)  \left(  w,s\right)  =\left(
v+t\diamond w,t+s\right)  $ where
\[
t\diamond w=\left[
\begin{array}
[c]{cc}%
e^{t_{2}} & t_{1}e^{t_{2}}\\
0 & e^{t_{2}}%
\end{array}
\right]  \left[
\begin{array}
[c]{c}%
w_{1}\\
w_{2}%
\end{array}
\right]  .
\]
For a fixed linear functional $\lambda$ of $\mathbb{R}^{2},$ let
$\pi=\mathrm{ind}_{\mathbb{R}^{2}}^{\mathbb{R}^{2}\rtimes\mathbb{R}^{2}%
}\left(  \chi_{\lambda}\right)  $ be a unitary representation of $G$ realized
as acting on $L^{2}\left(  \mathbb{R}^{2}\right)  $ as follows. For a
square-integrable function $\mathbf{f}$ over $\mathbb{R}^{2},$ the action of
$\pi$ is described as follows
\[
\left[  \pi\left(  v,s\right)  \mathbf{f}\right]  \left(  t\right)  =e^{2\pi
i\left\langle \left[
\begin{array}
[c]{c}%
\lambda_{1}e^{-t_{2}}\\
e^{-t_{2}}\left(  \lambda_{2}-\lambda_{1}t_{1}\right)
\end{array}
\right]  ,%
\begin{array}
[c]{c}%
v_{1}\\
v_{2}%
\end{array}
\right\rangle }\cdot\mathbf{f}\left(  t-s\right)  \text{ for }\left(
v,t\right)  \in\mathbb{R}^{2}\rtimes\mathbb{R}^{2}.
\]
Take $\lambda_{1}=1$ and $\lambda_{2}=0$, and $J$ to be the set containing $1$
and $2.$ By identifying $G$ with its Lie algebra, $\varphi$ is just the
identity map. Next, let $\Theta_{\lambda}:%
\mathbb{R}
^{2}\rightarrow\left(  0,\infty\right)  \times%
\mathbb{R}
$ such that $\Theta_{\lambda}\left(  t_{1},t_{2}\right)  =\left(  e^{-t_{2}%
},-e^{-t_{2}}t_{1}\right)  .$ The inverse of $\Theta_{\lambda}$ is given by
$\Theta_{\lambda}^{-1}:\left(  0,\infty\right)  \times%
\mathbb{R}
\rightarrow%
\mathbb{R}
^{2}$ such that $\Theta_{\lambda}^{-1}\left(  \xi_{1},\xi_{2}\right)  =\left(
-\frac{\xi_{2}}{\xi_{1}},\ln\left(  \frac{1}{\xi_{1}}\right)  \right)  $ and
the Jacobian of $\Theta_{\lambda}^{-1}$ is
\[
\mathrm{Jac}_{\Theta_{\lambda}^{-1}}\left(  \xi_{1},\xi_{2}\right)  =\left[
\begin{array}
[c]{cc}%
\xi_{2}\xi_{1}^{-2} & -\xi_{1}^{-1}\\
-\xi_{1}^{-1} & 0
\end{array}
\right]  .
\]
Since $H$ is commutative and is identified with its Lie algebra,
\[
\mathbf{W}_{\lambda}\left(  \xi_{1},\xi_{2}\right)  =\left\vert \det
\mathrm{Jac}_{\Theta_{\lambda}^{-1}}\left(  \xi_{1},\xi_{2}\right)
\right\vert =\xi_{1}^{-2}.
\]
Next, $\Theta_{\lambda}$ defines a global diffeomorphism between $%
\mathbb{R}
^{2}$ and $\left(  0,\infty\right)  \times%
\mathbb{R}
.$ According to Theorem \ref{main} for any continuous function $\mathbf{f}$
supported on some compact set $K\subset%
\mathbb{R}
^{2}$, there exists a discrete set $\Gamma_{K,\mathbf{f}}\subset G$ such that
$\left\{  \pi\left(  \gamma\right)  \mathbf{f}:\gamma\in\Gamma_{K,\mathbf{f}%
}\right\}  $ is a frame for $L^{2}\left(
\mathbb{R}
^{2},dx\right)  .$
\end{example}

\begin{example}
\label{SL2}(A case where $H$ is $SL\left(  2,%
\mathbb{R}
\right)  $) Following the construction given in Proposition \ref{embed}, let
\[
G=\left\{  \left[
\begin{array}
[c]{cccc}%
1 & 0 & x_{1} & x_{2}\\
0 & 1 & x_{3} & x_{4}\\
0 & 0 & 1 & 0\\
0 & 0 & 0 & 1
\end{array}
\right]  \left[
\begin{array}
[c]{cccc}%
a\cos\theta & as\cos\theta-\frac{1}{a}\sin\theta & 0 & 0\\
a\sin\theta & \frac{1}{a}\cos\theta+as\sin\theta & 0 & 0\\
0 & 0 & 1 & 0\\
0 & 0 & 0 & 1
\end{array}
\right]  :x_{1},\cdots,x_{4},\theta,a,s\in%
\mathbb{R}
\right\}
\]
be a matrix group with Lie algebra $\mathfrak{g=n+h}$ such that $\mathfrak{n}$
is spanned by
\[
\left[
\begin{array}
[c]{cccc}%
0 & 0 & 1 & 0\\
0 & 0 & 0 & 0\\
0 & 0 & 0 & 0\\
0 & 0 & 0 & 0
\end{array}
\right]  ,\left[
\begin{array}
[c]{cccc}%
0 & 0 & 0 & 1\\
0 & 0 & 0 & 0\\
0 & 0 & 0 & 0\\
0 & 0 & 0 & 0
\end{array}
\right]  ,\left[
\begin{array}
[c]{cccc}%
0 & 0 & 0 & 0\\
0 & 0 & 1 & 0\\
0 & 0 & 0 & 0\\
0 & 0 & 0 & 0
\end{array}
\right]  ,\left[
\begin{array}
[c]{cccc}%
0 & 0 & 0 & 0\\
0 & 0 & 0 & 1\\
0 & 0 & 0 & 0\\
0 & 0 & 0 & 0
\end{array}
\right]
\]
and $\mathfrak{h}$ is spanned by the matrices
\[
\left[
\begin{array}
[c]{cccc}%
0 & -1 & 0 & 0\\
1 & 0 & 0 & 0\\
0 & 0 & 0 & 0\\
0 & 0 & 0 & 0
\end{array}
\right]  ,\left[
\begin{array}
[c]{cccc}%
0 & 1 & 0 & 0\\
0 & 0 & 0 & 0\\
0 & 0 & 0 & 0\\
0 & 0 & 0 & 0
\end{array}
\right]  ,\left[
\begin{array}
[c]{cccc}%
1 & 0 & 0 & 0\\
0 & -1 & 0 & 0\\
0 & 0 & 0 & 0\\
0 & 0 & 0 & 0
\end{array}
\right]  .
\]
Let $\varphi:SL\left(  2,%
\mathbb{R}
\right)  =H\rightarrow%
\mathbb{R}
^{3}$ such that
\[
\varphi\left(  \left[
\begin{array}
[c]{cccc}%
a\cos\theta & as\cos\theta-\frac{1}{a}\sin\theta & 0 & 0\\
a\sin\theta & \frac{1}{a}\cos\theta+as\sin\theta & 0 & 0\\
0 & 0 & 1 & 0\\
0 & 0 & 0 & 1
\end{array}
\right]  \right)  =\left(  \theta,a,s\right)  .
\]
Then $\varphi$ defines a local diffeomorphism at the identity of $H$ and there
exists an open set $\mathcal{O}$ around the identity of $H$ such that $\left(
\mathcal{O},\varphi\right)  $ is a smooth chart. Next, we fix a linear
functional
\[
\lambda=\left[
\begin{array}
[c]{cccc}%
0 & 0 & 1 & 0\\
0 & 0 & 0 & 1\\
0 & 0 & 0 & 0\\
0 & 0 & 0 & 0
\end{array}
\right]  \in\mathfrak{n}^{\ast}.
\]
To compute $Ad\left(  h^{-1}\right)  ^{\ast}\lambda,$ we proceed as follows.
Let $M\left(  \theta,a,s\right)  $ be the inverse transpose of the following
matrix.
\[
\left[
\begin{array}
[c]{cccc}%
\cos\theta & -\sin\theta & 0 & 0\\
\sin\theta & \cos\theta & 0 & 0\\
0 & 0 & 1 & 0\\
0 & 0 & 0 & 1
\end{array}
\right]  \left[
\begin{array}
[c]{cccc}%
a & 0 & 0 & 0\\
0 & 1/a & 0 & 0\\
0 & 0 & 1 & 0\\
0 & 0 & 0 & 1
\end{array}
\right]  \left[
\begin{array}
[c]{cccc}%
1 & s & 0 & 0\\
0 & 1 & 0 & 0\\
0 & 0 & 1 & 0\\
0 & 0 & 0 & 1
\end{array}
\right]  .
\]
With straightforward calculations, we obtain
\[
Ad\left(  h^{-1}\right)  ^{\ast}\lambda=M\left(  \theta,a,s\right)
\lambda=\left[
\begin{array}
[c]{cccc}%
0 & 0 & a^{-1}\cos\theta+as\sin\theta & -a\sin\theta\\
0 & 0 & a^{-1}\sin\theta-as\cos\theta & a\cos\theta\\
0 & 0 & 0 & 0\\
0 & 0 & 0 & 0
\end{array}
\right]  .
\]
The map $h\mapsto Ad\left(  h^{-1}\right)  ^{\ast}\lambda$ is clearly an
immersion at the identity of $H.$ Next, let $P^{\ast}:\mathfrak{n}^{\ast
}\rightarrow\mathfrak{n}^{\ast}$ be a linear projection satisfying
\[
P^{\ast}\left(  \left[
\begin{array}
[c]{cccc}%
0 & 0 & \frac{1}{a}\cos\theta+as\sin\theta & -a\sin\theta\\
0 & 0 & \frac{1}{a}\sin\theta-as\cos\theta & a\cos\theta\\
0 & 0 & 0 & 0\\
0 & 0 & 0 & 0
\end{array}
\right]  \right)  =\left[
\begin{array}
[c]{cccc}%
0 & 0 & 0 & -a\sin\theta\\
0 & 0 & \frac{1}{a}\sin\theta-as\cos\theta & a\cos\theta\\
0 & 0 & 0 & 0\\
0 & 0 & 0 & 0
\end{array}
\right]  .
\]
With respect to this projection, we define $\Theta_{\lambda}$ (in local
coordinates)\ as follows
\[
\Theta_{\lambda}\left(  \theta,a,s\right)  =\left(  \frac{-a^{2}s\cos
\theta+\sin\theta}{a},-a\sin\theta,a\cos\theta\right)  .
\]
The Jacobian of $\Theta_{\lambda}$ is given by
\[
\left[
\begin{array}
[c]{ccc}%
\frac{1}{a}\cos\theta+as\sin\theta & -s\cos\theta-\frac{1}{a^{2}}\sin\theta &
-a\cos\theta\\
-a\cos\theta & -\sin\theta & 0\\
-a\sin\theta & \cos\theta & 0
\end{array}
\right]
\]
and its determinant $a^{2}\cos\theta$ is nonzero at $\left(  \theta
,a,s\right)  =\left(  0,1,0\right)  $. Thus, $\Theta_{\lambda}$ is a local
diffeomorphism at $\left(  0,1,0\right)  $. Using the following Haar measure
$a^{-3}d\theta dads$ on $H,$ we obtain
\[
\mathbf{W}_{\lambda}\left(  \xi\right)  =\frac{1}{\left\vert \xi
_{3}\right\vert \cdot\left(  \xi_{2}^{2}+\xi_{3}^{2}\right)  ^{2}}.
\]
Let $\mathcal{O}$ be a relatively compact subset of $SL\left(  2,%
\mathbb{R}
\right)  $ such that $\Theta_{\lambda}$ defines a diffeomorphism between
$\mathcal{O}$ and its image. According to Theorem \ref{main}, for any
continuous function $\mathbf{f}$ supported on $\mathcal{O}$, there exists a
discrete set $\Gamma_{\mathbf{f,}\mathcal{O}}\subset G$ such that $\left\{
\pi\left(  \gamma\right)  \mathbf{f}:\gamma\in\Gamma_{\mathbf{f,}\mathcal{O}%
}\right\}  $ is a frame for $L^{2}\left(  SL\left(  2,%
\mathbb{R}
\right)  \right)  .$
\end{example}

\subsection{Construction of orthonormal bases}

We present below some examples illustrating Proposition \ref{criteria}. Our
first example shares a striking resemblance with Gabor analysis.

\begin{example}
\label{ONB}Let
\[
\mathfrak{n}=\sum_{k=1}^{3}\mathbb{R}X_{k},\mathfrak{h}=\sum_{k=1}%
^{2}\mathbb{R}A_{k}%
\]
such that
\[
\sum_{k=1}^{3}x_{k}X_{k}=\left[
\begin{array}
[c]{cccc}%
0 & 0 & 0 & x_{1}\\
0 & 0 & 0 & x_{2}\\
0 & 0 & 0 & x_{3}\\
0 & 0 & 0 & 0
\end{array}
\right]  \text{ and }\sum_{k=1}^{2}a_{k}A_{k}=-\left[
\begin{array}
[c]{cccc}%
0 & a_{1} & a_{2} & 0\\
0 & 0 & a_{1} & 0\\
0 & 0 & 0 & 0\\
0 & 0 & 0 & 0
\end{array}
\right]  .
\]
Every element of $G=NH$ can be uniquely factored as
\[
\exp\left(  \sum_{k=1}^{3}x_{k}X_{k}\right)  \exp\left(  \sum_{k=1}^{2}%
a_{k}A_{k}\right)  =\left[
\begin{array}
[c]{cccc}%
1 & 0 & 0 & x_{1}\\
0 & 1 & 0 & x_{2}\\
0 & 0 & 1 & x_{3}\\
0 & 0 & 0 & 1
\end{array}
\right]  \left[
\begin{array}
[c]{cccc}%
1 & -a_{1} & \frac{1}{2}a_{1}^{2}-a_{2} & 0\\
0 & 1 & -a_{1} & 0\\
0 & 0 & 1 & 0\\
0 & 0 & 0 & 1
\end{array}
\right]  .
\]
In fact $G$ is a step-three metabelian nilpotent Lie group which was also
studied in \cite{oussa2017regular} in the context of sampling theory for
left-invariant spaces. Next, let $\lambda=X_{1}^{\ast}$ \ and define the
linear projection $P:\mathfrak{n}\rightarrow\mathfrak{n}$ of rank two such
that
\[
P\left(  \sum_{k=1}^{3}x_{k}X_{k}\right)  =\sum_{k=2}^{3}x_{k}X_{k}.
\]
Note that the center of the algebra $\mathfrak{g}=\mathfrak{n}+\mathfrak{h}$
is contained in the null-space of $P$ and
\[
\Theta_{\lambda}\left(  t_{1},t_{2}\right)  =\left[
\begin{array}
[c]{ccc}%
0 & 0 & 0\\
0 & 1 & 0\\
0 & 0 & 1
\end{array}
\right]  \left[
\begin{array}
[c]{ccc}%
1 & 0 & 0\\
t_{1} & 1 & 0\\
\frac{1}{2}t_{1}^{2}+t_{2} & t_{1} & 1
\end{array}
\right]  \left[
\begin{array}
[c]{c}%
1\\
0\\
0
\end{array}
\right]  =\left[
\begin{array}
[c]{c}%
0\\
t_{1}\\
\frac{1}{2}t_{1}^{2}+t_{2}%
\end{array}
\right]  .
\]
To simplify our presentation, we just write
\[
\Theta_{\lambda}\left(  t_{1},t_{2}\right)  =\left[  t_{1},\frac{1}{2}%
t_{1}^{2}+t_{2}\right]  ^{T}.
\]
It is easy to verify that $\Theta_{\lambda}$ defines a diffeomorphism between
$%
\mathbb{R}
^{2}$ and its image. In fact, since
\[
\det\left[
\begin{array}
[c]{cc}%
\dfrac{\partial\left(  t_{1}\right)  }{\partial t_{1}} & \dfrac{\partial
\left(  t_{1}\right)  }{\partial t_{2}}\\
\dfrac{\partial\left(  \frac{1}{2}t_{1}^{2}+t_{2}\right)  }{\partial t_{1}} &
\dfrac{\partial\left(  \frac{1}{2}t_{1}^{2}+t_{2}\right)  }{\partial t_{2}}%
\end{array}
\right]  =\det\left[
\begin{array}
[c]{cc}%
1 & 0\\
t_{1} & 1
\end{array}
\right]  =1,
\]
it follows that $\mathbf{W}_{\lambda}\left(  \xi\right)  =1.$ Moreover, each
of the following sets: $\left[  -1/2,1/2\right)  ^{2}$ and $\Theta_{\lambda
}\left(  \left[  -1/2,1/2\right)  ^{2}\right)  $ tiles $%
\mathbb{R}
^{2}$ translationally by $%
\mathbb{Z}
^{2}$.

\includegraphics[scale=0.3]{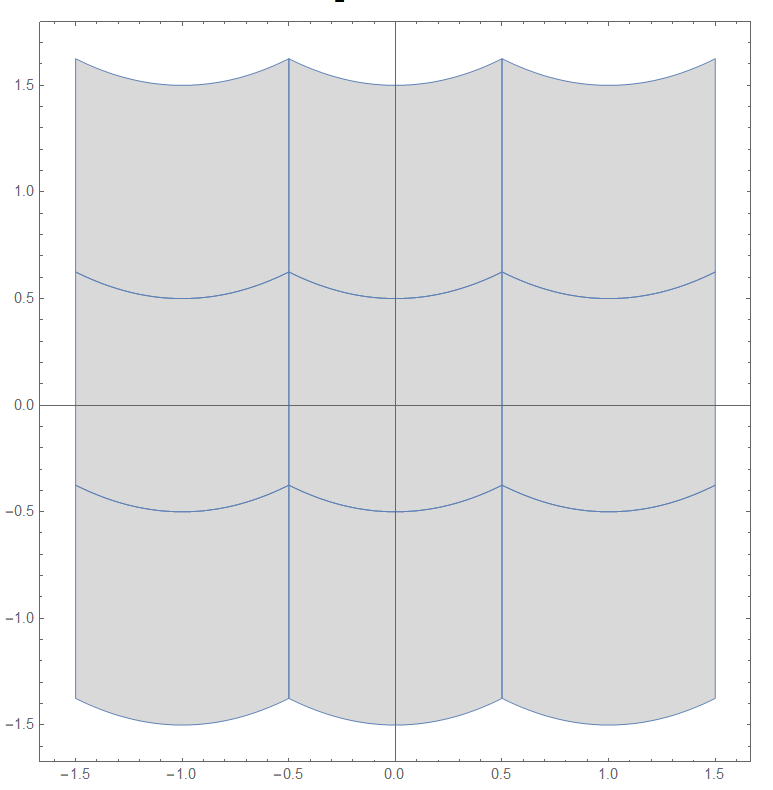} \newline Tiling the plane by integral
translation of $\Theta_{\lambda}\left(  \left[  -1/2,1/2\right)  ^{2}\right)
$

Put $\Gamma=\exp\left(  \sum_{k=1}^{2}\mathbb{Z}A_{k}\right)  \exp\left(
\sum_{k=2}^{3}\mathbb{Z}X_{k}\right)  .$ Appealing to Proposition
\ref{criteria}, we conclude that
\[
\left\{  \mathrm{ind}_{N}^{NH}\left(  \chi_{X_{1}^{\ast}}\right)  \left(
\gamma\right)  1_{\left[  -1/2,1/2\right)  ^{2}}:\gamma\in\Gamma\right\}
\]
is an orthonormal basis for $L^{2}\left(  \mathbb{R}^{2},dt\right)  .$
\end{example}

Our next example gives a construction of orthonormal bases on a class of
nilpotent Lie groups.

\begin{example}
Following the notation of Proposition \ref{embed}, we let
\[
G=\left\{  \left[
\begin{array}
[c]{cc}%
h & x\\
\mathrm{0}_{n} & \mathrm{Id}_{n}%
\end{array}
\right]  :h\in K\text{ and }x\in\mathfrak{gl}\left(  n,%
\mathbb{R}
\right)  \right\}
\]
such that
\[
K=\left\{  \left[
\begin{array}
[c]{cccc}%
1 & a_{11} & \cdots & a_{1n}\\
& 1 & \ddots & \vdots\\
&  & \ddots & a_{n1}\\
&  &  & 1
\end{array}
\right]  :\left(  a_{11},\cdots a_{n1}\right)  \in%
\mathbb{R}
^{\frac{n^{2}-n}{2}}\right\}  .
\]
Here $K$ is a simply connected nilpotent Lie group. In fact, the group $G$ is
also a simply connected nilpotent Lie group. Next, we consider the following
linear functional
\[
\lambda=\left[
\begin{array}
[c]{cc}%
\mathrm{0}_{n} & \mathrm{Id}_{n}\\
\mathrm{0}_{n} & \mathrm{0}_{n}%
\end{array}
\right]  \in\mathfrak{n}^{\ast}.
\]
The map $h\mapsto Ad\left(  h^{-1}\right)  ^{\ast}\lambda$ defines an
immersion at the identity of $H$ and the representation obtained by inducing
$\chi_{\lambda}$ is irreducible. Let $P^{\ast}:\mathfrak{n}^{\ast}%
\rightarrow\mathfrak{n}^{\ast}$ be the linear projection given by
\[
P^{\ast}\left(  \left[
\begin{array}
[c]{cc}%
\mathrm{0}_{n} & \omega\\
\mathrm{0}_{n} & \mathrm{0}_{n}%
\end{array}
\right]  \right)  =\left[
\begin{array}
[c]{cc}%
\mathrm{0}_{n} & \Pi\left(  \omega\right) \\
\mathrm{0}_{n} & \mathrm{0}_{n}%
\end{array}
\right]
\]
where
\[
\Pi\left(  \left[
\begin{array}
[c]{cccc}%
\lambda_{11} & \lambda_{12} & \cdots & \lambda_{1n}\\
\lambda_{21} & \lambda_{22} & \cdots & \lambda_{2n}\\
\vdots & \vdots & \ddots & \vdots\\
\lambda_{n1} & \lambda_{n2} & \cdots & \lambda_{nn}%
\end{array}
\right]  \right)  =\left[
\begin{array}
[c]{ccccc}%
0 &  &  &  & \\
\lambda_{21} & 0 &  &  & \\
\lambda_{31} & \cdots & 0 &  & \\
\vdots & \vdots & \vdots & \ddots & \\
\lambda_{n1} & \lambda_{n2} & \cdots & \lambda_{nn-1} & 0
\end{array}
\right]  .
\]
We describe $\Theta_{\lambda}:H\rightarrow\mathfrak{n}^{\ast}$ with respect to
the projection $P^{\ast}$ as follows
\[
\Theta_{\lambda}\left(  h\right)  =\left[
\begin{array}
[c]{cc}%
\mathrm{0}_{n} & \Pi\left(  \left(  h^{-1}\right)  ^{T}\right) \\
\mathrm{0}_{n} & \mathrm{0}_{n}%
\end{array}
\right]  .
\]
$\Theta_{\lambda}$ is a polynomial function in the coordinates of $H.$
Moreover, $\Theta_{\lambda}$ defines a global diffeomorphism between $H$ and
its image, and $\mathbf{W}_{\lambda}$ is merely the constant function $1$.
Appealing to the results in \cite{oussa2018frames}, it is not hard to verify
that there exist $C,\mathbf{f,}\Gamma_{H},\Gamma_{N}$ satisfying Conditions
\ref{Cond} (2) (6) (7) and (8). By Proposition \ref{criteria} Part (2), the
representation $\pi=\mathrm{ind}_{N}^{NH}\left(  \chi_{\lambda}\right)  $
admits an orthonormal basis of the type $\mathcal{S}\left(  \left\vert
C\right\vert ^{-1/2}\mathbf{f,}\Gamma_{H}^{-1}\Gamma_{N}\right)  $ for
$L^{2}\left(  H\right)  .$
\end{example}

\section{Concluding remarks}

For the special case where $NH$ is a nilpotent Lie group, Proposition
\ref{criteria} Part 2 gives sufficient conditions for the existence of
orthonormal bases generated by a function of the type $\mathbf{s}\cdot1_{A}.$
K. Gr\"{o}chenig and D. Rottensteiner recently \cite{grochenig2017orthonormal}
proved the following. Let $G$ be a graded Lie group (thus nilpotent Lie group)
with a one-dimensional center. Next, let $\pi$ be a square-integrable
irreducible representation of $G$ realized as acting in $L^{2}\left(
\mathbb{R}
^{d}\right)  $ which is square-integrable modulo its center. Then there exist
a set $\Gamma\subset G$ and a compact set $F$ such that
\[
\left\{  \left\vert F\right\vert ^{-1/2}\pi\left(  \gamma\right)  1_{F}%
:\gamma\in\Gamma\right\}
\]
is an orthonormal basis. We would also like to remark that there exist unitary
representations of nilpotent Lie groups which are not square-integrable modulo
the center which we can still discretize to construct orthonormal bases for
the corresponding Hilbert space. To see this, let $G$ be a six-dimensional
simply connected nilpotent Lie group with Lie algebra spanned by
$X_{23},X_{13},X_{12},X_{3},X_{2},X_{1}$ with corresponding dual vector space
spanned by the dual basis
\[
X_{23}^{\ast},X_{13}^{\ast},X_{12}^{\ast},X_{3}^{\ast},X_{2}^{\ast}%
,X_{1}^{\ast}%
\]
such that $\left\langle X_{k}^{\ast},X_{j}\right\rangle =\delta_{kj}$. The
non-trivial brackets of the Lie algebra are given by $\left[  X_{i}%
,X_{j}\right]  =X_{ij}$ for $i<j.$ Thus, the elements $X_{23},X_{13},X_{12}$
span the center of the Lie algebra of $G$ and $G$ is a step-two free nilpotent
Lie group on three generators. Fix a linear functional
\[
\lambda=\lambda_{23}X_{23}^{\ast}+\lambda_{13}X_{13}^{\ast}+\lambda_{12}%
X_{12}^{\ast}+\lambda_{3}X_{3}^{\ast}+\lambda_{2}X_{2}^{\ast}+\lambda_{1}%
X_{1}^{\ast}%
\]
such that $\lambda_{23}$ is nonzero, and consider a subalgebra $\mathfrak{n}$
spanned by the vectors $X_{23},X_{13},X_{12},\lambda_{12}X_{3}-\lambda
_{13}X_{2}+\lambda_{23}X_{1}.$ Then $\mathfrak{n}$ is a maximal algebra of
$\mathfrak{g}$ such that $\left[  \mathfrak{n},\mathfrak{n}\right]  $ is
contained in the kernel of the linear functional $\lambda.$ $\mathfrak{n}$ is
called a polarizing algebra subordinated to the linear functional $\lambda$
\cite{oussa2015computing}. Next, we consider the corresponding irreducible
representation $\pi_{\lambda}$ of $G$ realized as acting in $L^{2}\left(
\mathbb{R}\right)  $ as follows
\begin{equation}
\pi_{\lambda}\left(  \exp tX\right)  f\left(  x\right)  =%
\begin{cases}
f\left(  x-t\right)  & \text{ if }X=X_{2}\\
e^{-2\pi i\lambda_{23}xt}f\left(  x\right)  & \text{ if }X=X_{3}\\
e^{2\pi it\lambda_{1}}e^{2\pi ixt\lambda_{12}}e^{-2\pi i\frac{t^{2}%
\lambda_{12}\lambda_{13}}{\lambda_{23}}}f\left(  x-\frac{t\lambda_{12}%
}{\lambda_{23}}\right)  & \text{ if }X=X_{1}\\
e^{2\pi it\lambda_{ij}}f\left(  x\right)  & \text{ if }X=X_{ij}%
\end{cases}
.
\end{equation}
Let
\[
\mathfrak{r}_{\lambda}=\left\{  X\in\mathfrak{g}:ad\left(  X\right)  ^{\ast
}\lambda=0\right\}
\]
be the so-called radical of $\lambda.$ Then $\mathfrak{r}_{\lambda}$ contains
the center $\mathfrak{z}\left(  \mathfrak{g}\right)  $ of $\mathfrak{g}$ and
since
\[
\dim\left(  \mathfrak{r}_{\lambda}\right)  =4>3=\dim\left(  \mathfrak{z}%
\left(  \mathfrak{g}\right)  \right)  ,
\]
according to \cite[4.5.4 Corollary]{corwin2004representations}), $\pi
_{\lambda}$ is not square-integrable modulo the center of $G$. Although, the
action given by $\pi_{\lambda}\left(  \exp X_{1}\right)  $ is quite
complicated, we do not need to take it into account for the construction of
orthonormal bases. To this end, it suffices to select our discrete set to be
$\Gamma=\exp\left(  \lambda_{23}^{-1}\mathbb{Z}X_{3}\right)  \exp\left(
\mathbb{Z}X_{2}\right)  .$ Given $k,j\in%
\mathbb{Z}
$ and $\mathbf{f}\in L^{2}\left(  \mathbb{R}\right)  ,$ we have
\[
\exp\left(  \lambda_{23}^{-1}kX_{3}\right)  \mathbf{f}\left(  x\right)
=e^{-2\pi ikx}f\left(  x\right)  \text{ \ \ \ }\exp\left(  jX_{2}\right)
\mathbf{f}\left(  x\right)  =\mathbf{f}\left(  x-j\right)  .
\]
and the collection $\left\{  \pi_{\lambda}\left(  \gamma\right)  1_{\left[
0,1\right]  }:\gamma\in\Gamma\right\}  $ is an orthonormal basis for
$L^{2}\left(  \mathbb{R}\right)  .$ Note that in this example, Theorem
\ref{main} does not even apply since $G$ cannot be written as a semi-direct
product of the type $NH$ for some closed subgroups $N,H$ satisfying our assumptions.

We intend to generalize these constructions in a future investigation.

\section*{Acknowledgments}

I thank Gestur Olafsson for pointing me toward the book of J. Wolf which then
inspired me to state and prove Proposition \ref{embed}. I am also grateful for
Hartmut F\"{u}hr for sharing with me a complete proof of Lemma
\ref{communicated}.

\end{document}